\documentclass[11pt]{article}

\usepackage{amssymb, amsmath, amsthm}
\usepackage{verbatim}
\usepackage{multicol}
\usepackage{geometry}

\usepackage{mathtools}

\usepackage[colorlinks=true,urlcolor=blue]{hyperref}

\usepackage{graphicx}
\DeclareGraphicsExtensions{.pdf,.png,.jpg}

\usepackage{pgf,tikz}
\usetikzlibrary{arrows.meta}

\newcounter{reminder}

\DeclareMathOperator{\GCD}{GCD}

\DeclareMathOperator{\Aut}{Aut}
\DeclareMathOperator{\Norm}{Norm}
\DeclareMathOperator{\Sym}{Sym}
\DeclareMathOperator{\I}{Im}
\DeclareMathOperator{\R}{Re}

\theoremstyle{theorem}
\newtheorem{theorem}{Theorem}[section]

\newtheorem{proposition}[theorem]{Proposition}

\newtheorem*{MainTheorem}{Main Theorem}

\theoremstyle{definition}

\theoremstyle{remark}
\newtheorem*{remark}{\textsc{Remark}}

\newcommand{\calM}{\mathcal{M}}

\newcommand{\calN}{\mathcal{N}}

\title{Archimedean toroidal maps and their minimal almost regular covers}

\author{Kostiantyn Drach \and Yurii Haidamaka \and Mark Mixer \and Maksym Skoryk}

\begin{document}

\maketitle

\begin{abstract}
The automorphism group of a map acts naturally on its flags (triples of incident vertices, edges, and faces).  An Archimedean map on the torus is called almost regular if it has as few flag orbits as possible for its type; for example, a map of type $(4.8^2)$ is called almost regular if it has exactly three flag orbits.  Given a map of a certain type, we will consider other more symmetric maps that cover it.   In this paper, we prove that each Archimedean toroidal map has a unique minimal almost regular cover. By using the Gaussian and Eisenstein integers, along with previous results regarding equivelar maps on the torus, we construct these minimal almost regular covers explicitly.
\end{abstract}

\section{Introduction}

Throughout the last few decades there have been many results about polytopes and maps that are highly symmetric, but that are not necessarily regular.  In particular, there has been recent interest in the study of discrete objects using combinatorial, geometric, and algebraic approaches, with the topic of symmetries of maps receiving a lot of interest.  

There is a great history of work surrounding maps on the Euclidean plane or on the 2-dimensional torus.   When working with discrete symmetric structures on a torus, many of the ideas follow the concepts introduced by Coxeter and Moser in~\cite{Coxeter_Conf_maps, CM}, where they present a classification of rotary (regular and chiral) maps on the torus. 
Such ``toroidal'' maps can be seen as quotients of regular tessellations of the Euclidean plane.   More recently, Brehm and K\"uhnel~\cite{BrK} classified the equivelar maps on the two dimensional torus.  Several more results have appeared about highly symmetric maps (see~\cite{CD, TU} for example), and about highly symmetric tessellations of tori in larger dimensions (see~\cite{McMull,MSHigher}).  

There is also much interest in finding minimal regular covers of different families of maps and polytopes (see for example \cite{Hartley_Min_cov, Pellicer2, Pellicer1}). In a previous paper~\cite{DMEquivelar}, two of the authors constructed the minimal rotary cover of any equivelar toroidal map. Here we extend this idea to toroidal maps that are no longer equivelar, and construct minimal toroidal covers of the Archimedean toroidal maps with maximal symmetry.    We call these covers {\em almost regular}; they will no longer be regular (or chiral), but instead will have the same number of flag orbits as their associated tessellation of the Euclidean plane.

 Our main results can be summarized by the following theorem.  

\begin{MainTheorem}
Each Archimedean map on the torus has a unique almost regular cover on the torus, which can be constructed explicitly.
\end{MainTheorem}

The paper is organized as follows. Section~\ref{Sec:Prelim} contains the necessary background on maps and their symmetries, including the definition of an almost regular Archimedean map. In Section~\ref{Sec:AlmostReg}, almost regular Archimedean toroidal maps are characterized in terms of their lifts to the planar tessellations and the translation subgroups generating respective quotients. Section~\ref{Sec:MinCov} contains our main results (Theorems~\ref{Thm:Areg}--~\ref{Thm:33344}) regarding the relationship between maps and their minimal almost regular covers.

\section{Preliminares}
\label{Sec:Prelim}

In this section we provide definitions and results necessary for our main theorems; many of these ideas, as well as further details, can be found in~\cite{BrSch, McMull, Sir}.

A finite graph $ X$ embedded in a compact 2-dimensional manifold $S$ such that every connected component of $S\backslash X$ (which is called a \emph{face}) is homeomorphic to an open disc is called a \emph{map} $\mathcal M$ (on the surface $S$).

In this paper, we consider Archimedean maps on a flat 2-dimensional torus, which we call \textit{Archimedean toroidal maps}.   A map $\mathcal{M}$ on the torus is \textit{Archimedean} if the faces of $\mathcal{M}$ are regular polygons (in the canonical metric on $S$) and every two vertices of $\mathcal M$ can be mapped into each other by an automorphism of $S$.
 A map $\calM$ is equivelar of (Schl{\"a}fli) type $\{p,q\}$ if all of its vertices are $q$-valent, and all of its faces are regular $p$-gons.  If the map is Archimedean, it can be described, as in~\cite{Grunbaum}, by the arrangement of polygons around a vertex; where a map of type $(p_1.p_2\ldots p_k)$ has $k$ polygons (a $p_1$-gon, $p_2$-gon, \ldots , and a $p_k$-gon) incident to each vertex. A particular vertex structure of a map is called a \emph{type}.

As we will see below, Archimedean toroidal maps arise naturally as quotients of tessellations of the Euclidean plane by regular polygons; these tessellations are called \emph{Archimedean}, and they are described in the same way as Archimedean maps. The following classical theorem gives a complete classification of planar Archimedean tessellations.

\begin{theorem}[Classification of Archimedean tessellations on the plane; \cite{Grunbaum}]
There are only 11 tessellations on the plane by regular polygons so that any vertex can be mapped to every other vertex by the symmetry of the tessellation. These are the following tessellations: 
\begin{equation*}
\begin{aligned}
\{3,6\},\,\{4,4\},\,\{6,3\},\,(4.8^2),\,&(3.12^2),\,(3.6.3.6),\,(3.4.6.4), \\
(4.6.12),\,(3^2.4.3.4),\,&(3^4.6),\,(3^3.4^2) \\
\end{aligned}
\end{equation*}
(see Figures~\ref{63}--\ref{33344}). \qed
\end{theorem}

\begin{figure}[h]
\begin{multicols}{3}
\begin{center}
\includegraphics[width=30mm]{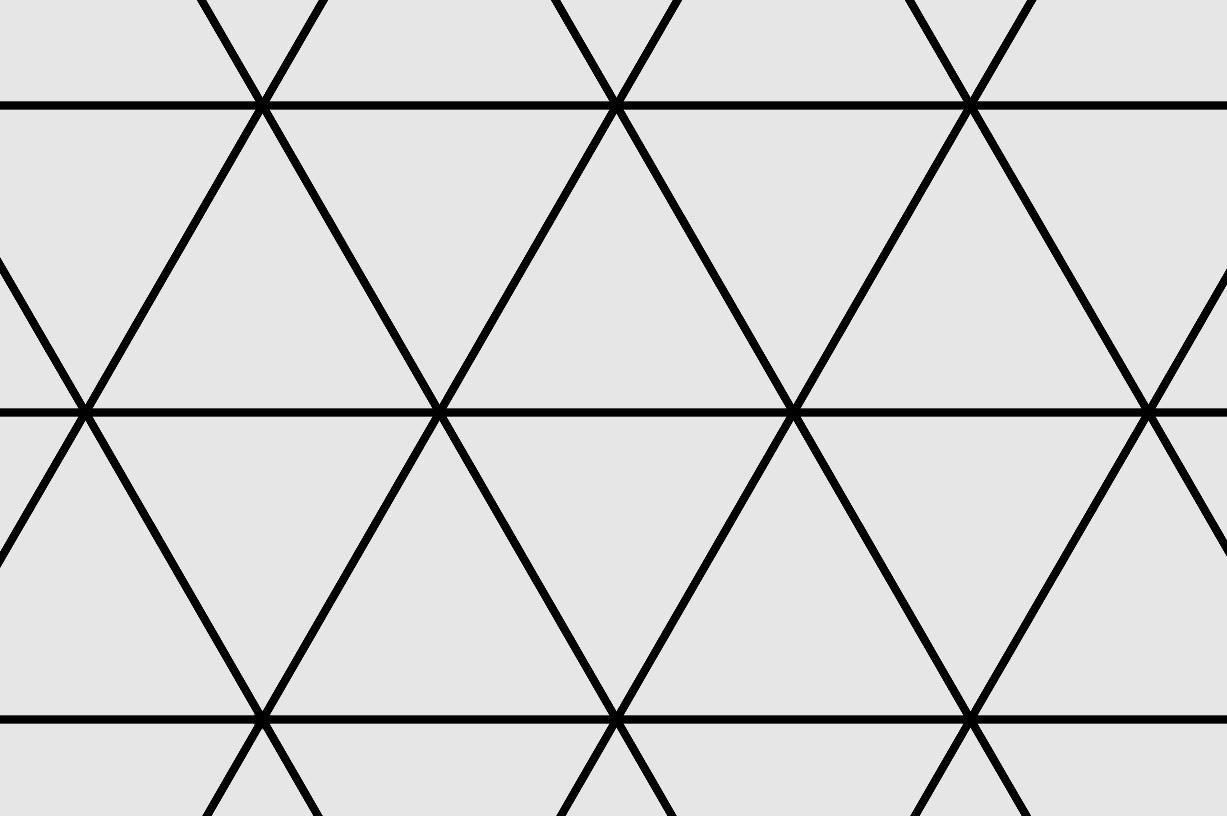}
\caption{$\{3,6\}$}
\label{63}
\end{center}
\begin{center}
\includegraphics[width=30mm]{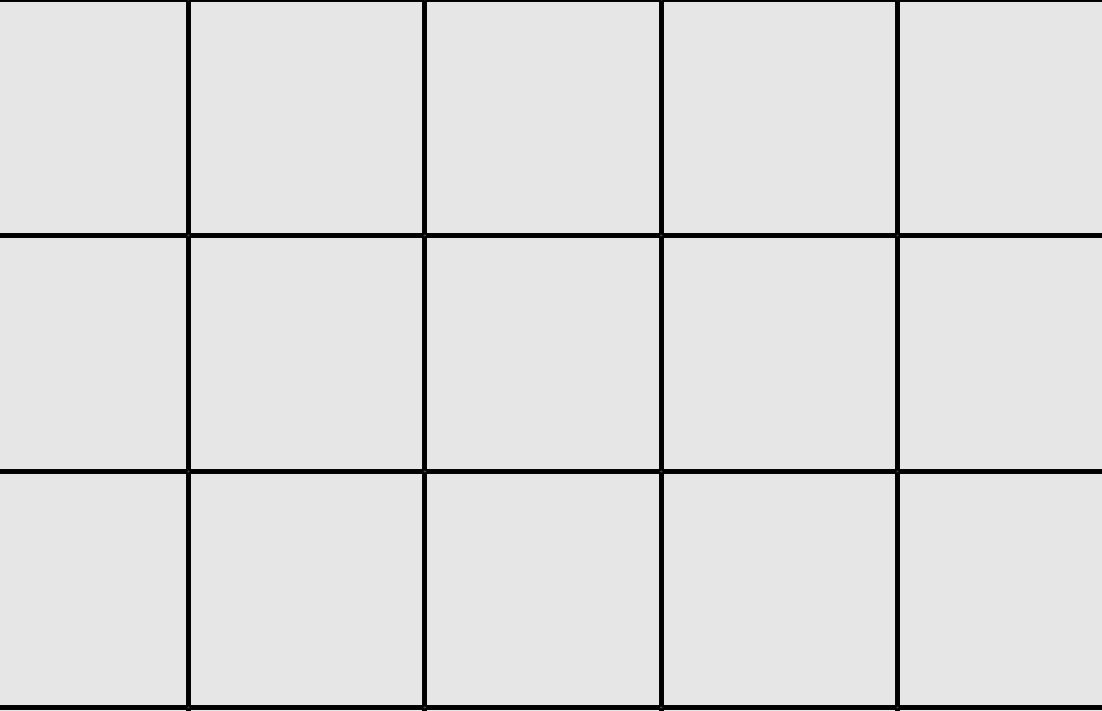}
\caption{$\{4,4\}$}
\label{figRight}
\label{44}
\end{center}
\begin{center}
\includegraphics[width=30mm]{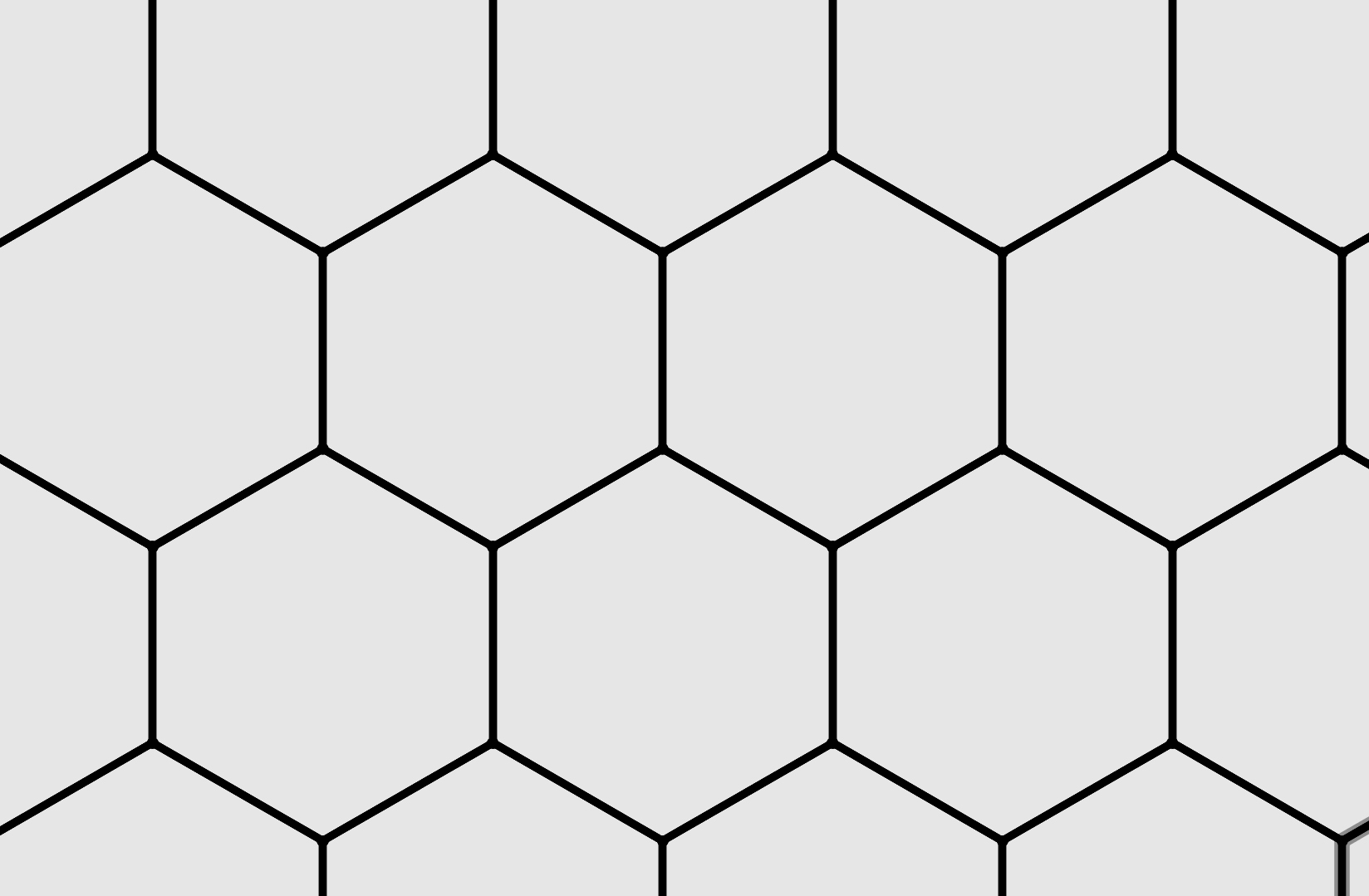}
\caption{$\{6,3\}$}
\label{36}		
\end{center}
\end{multicols}
\end{figure}

\begin{figure}[h!]
\begin{multicols}{3}
\begin{center}
\includegraphics[width=30mm]{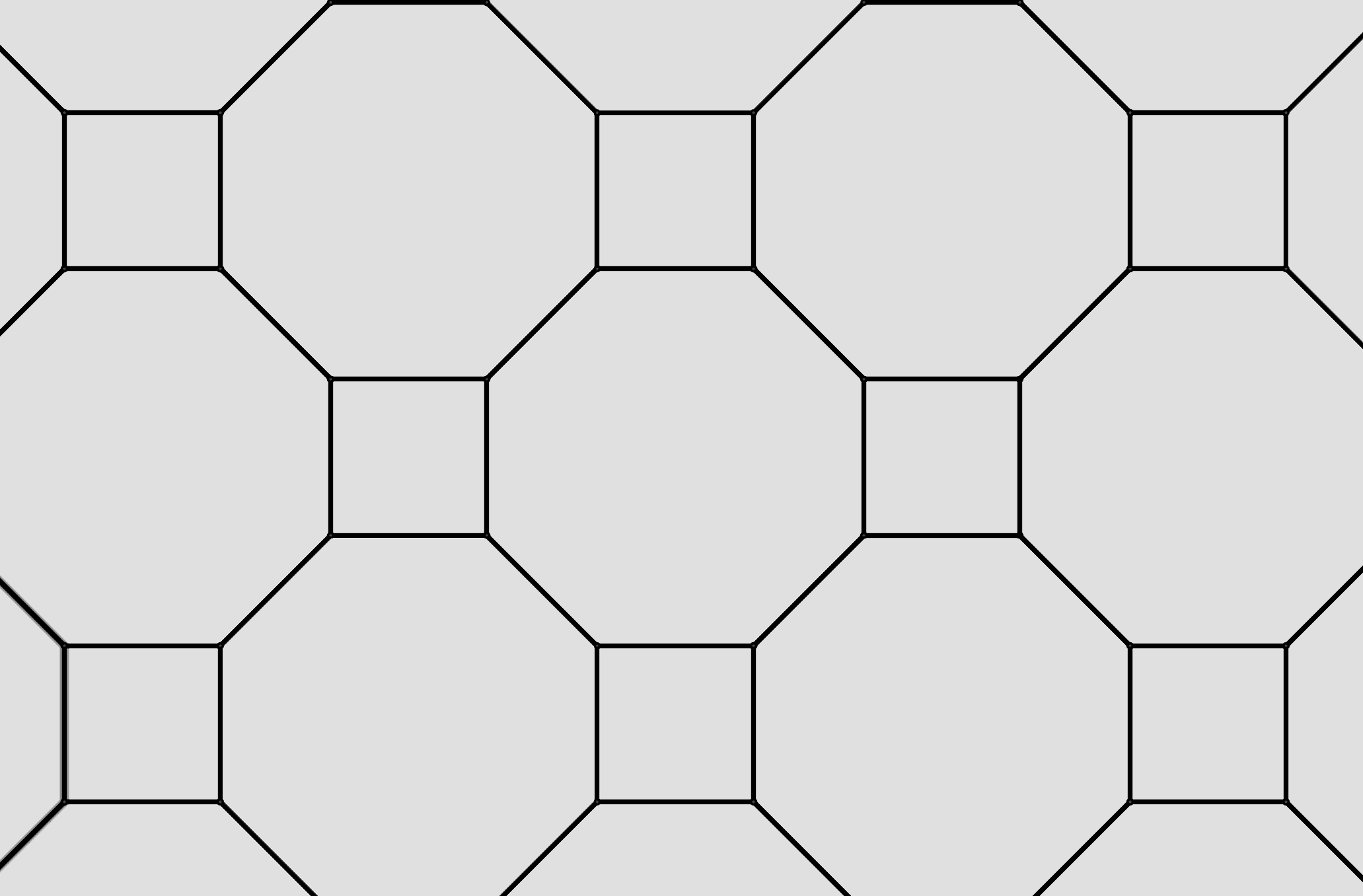}
\caption{$(4.8^2)$}
\label{488}
\end{center}
\begin{center}		
\includegraphics[width=30mm]{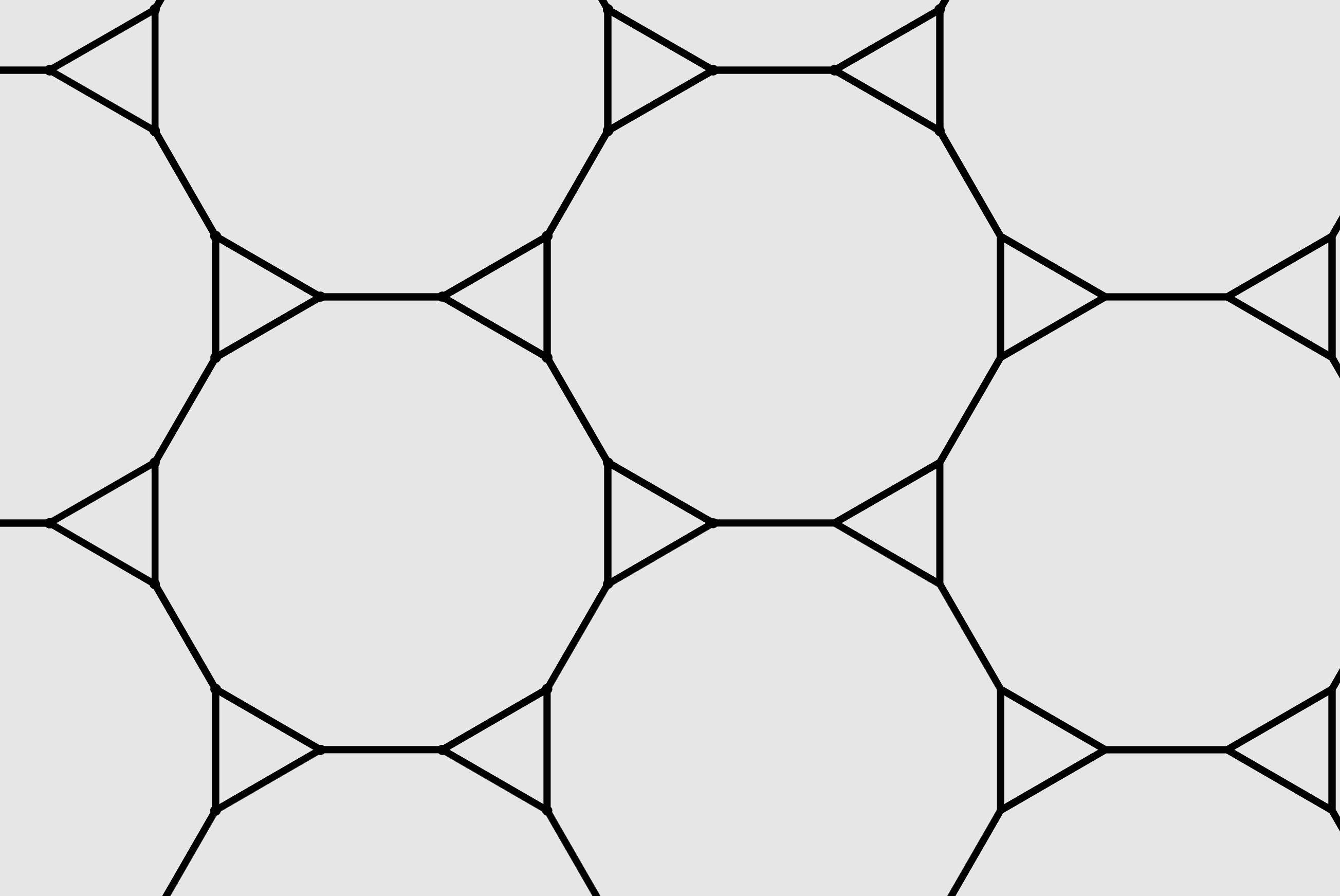}
\caption{$(3.12^2)$}
\label{31212}
\end{center}
\begin{center}
\includegraphics[width=30mm]{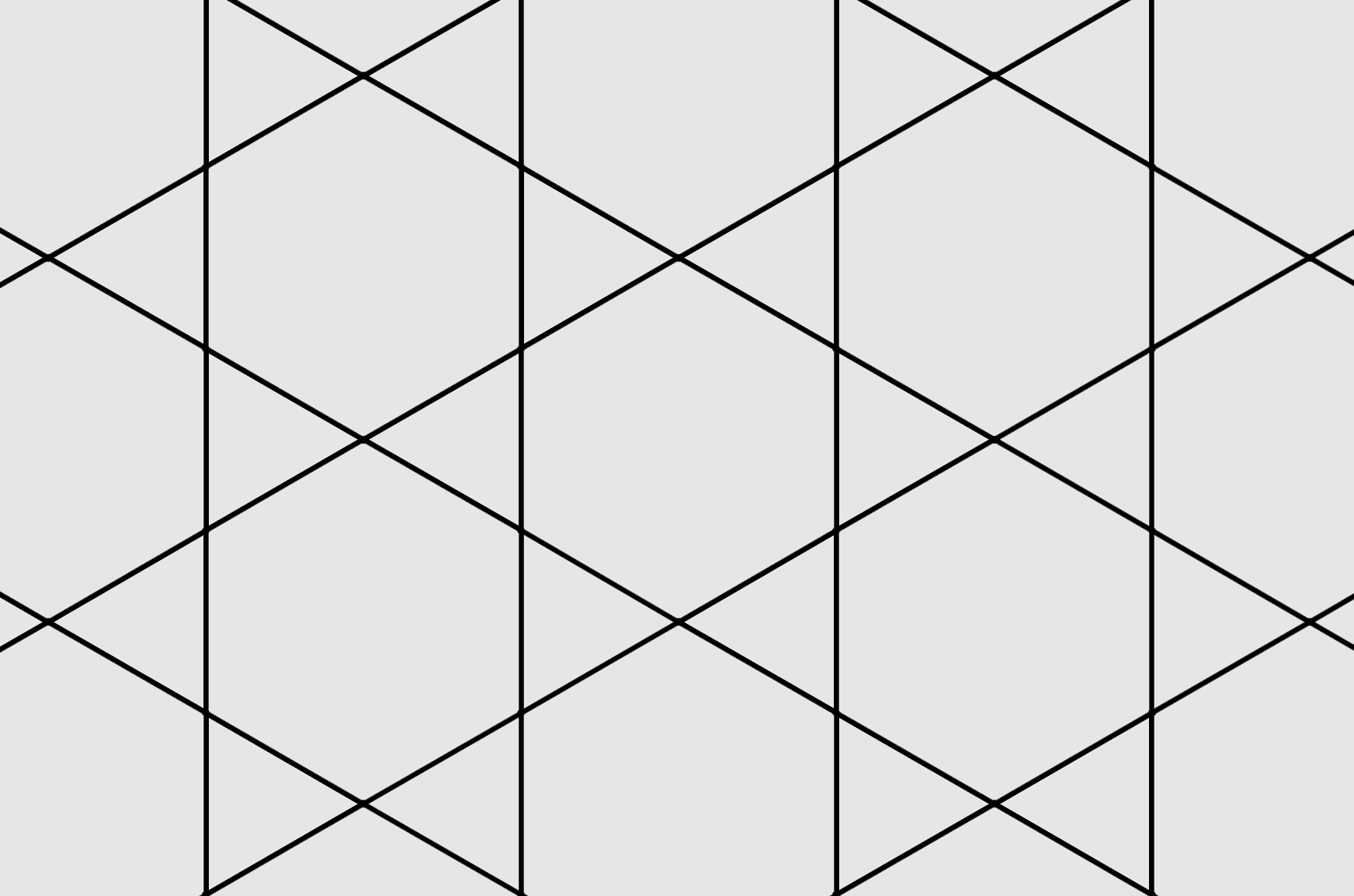}
\caption{$(3.6.3.6)$}
\label{3636}
\end{center}
\end{multicols}
\end{figure}

\begin{figure}[h!]
\begin{multicols}{3}
\begin{center}		
\includegraphics[width=30mm]{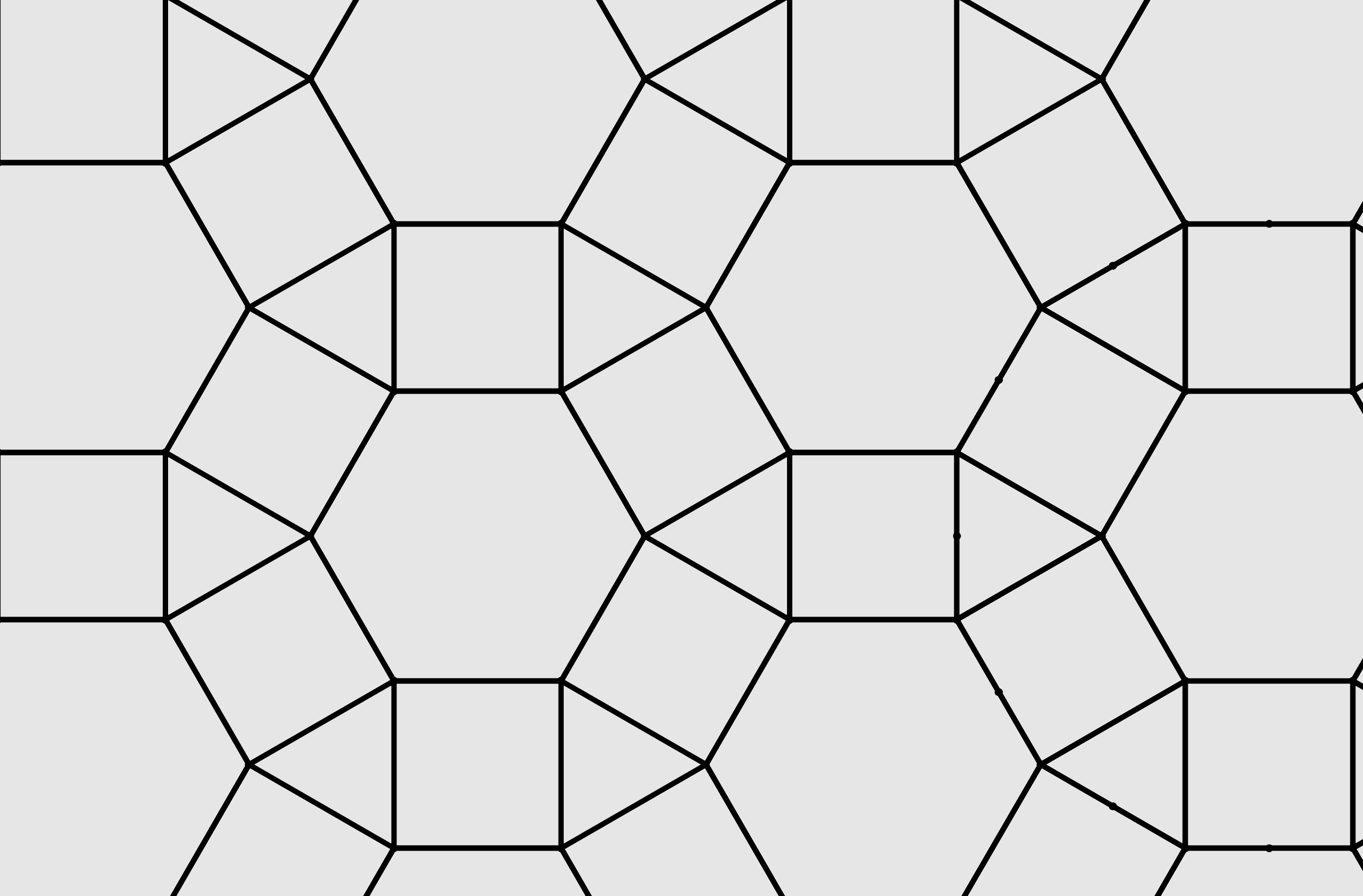}
\caption{$(3.4.6.4)$}
\label{3464}
\end{center}
\begin{center}		
\includegraphics[width=30mm]{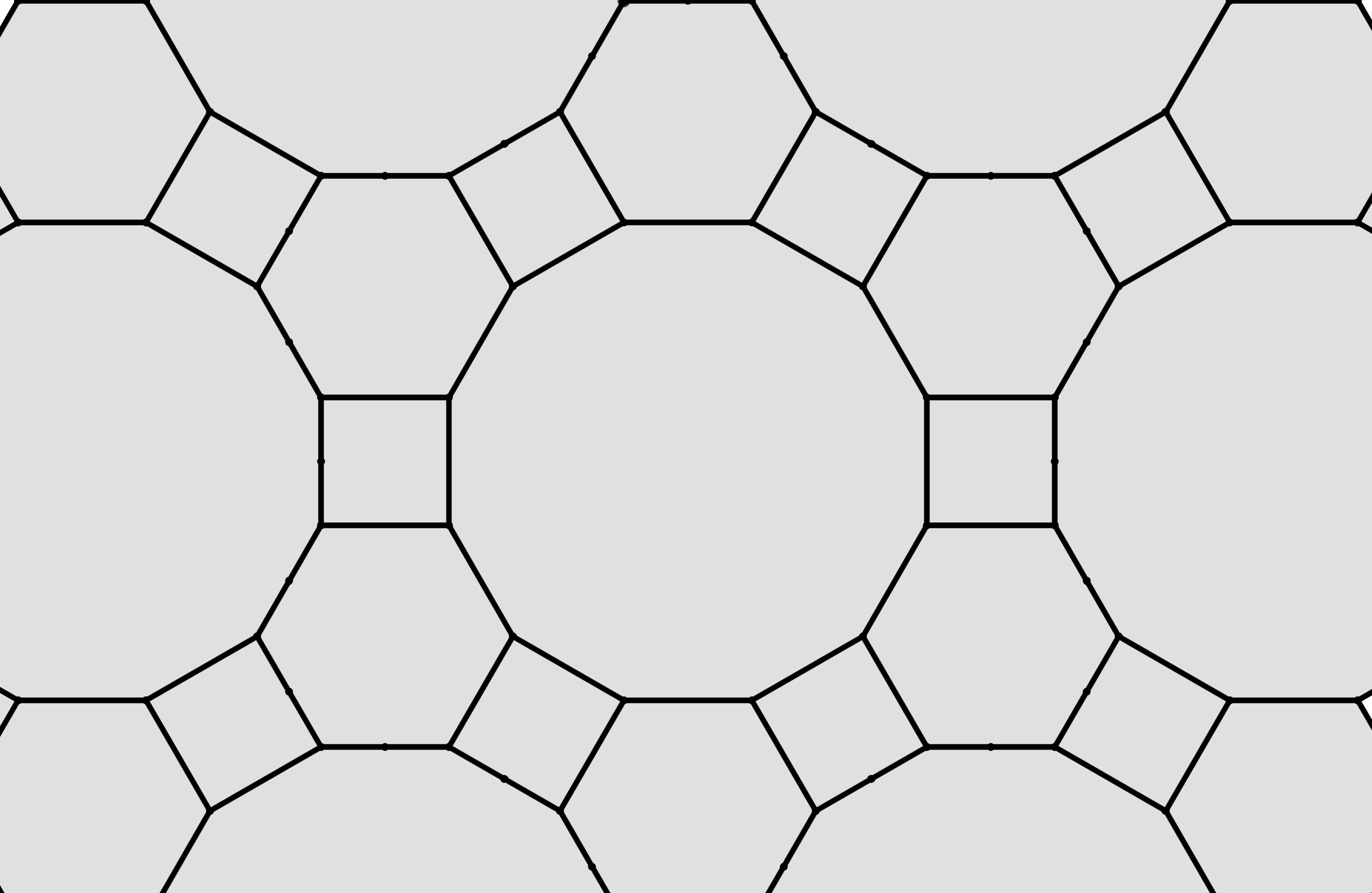}
\caption{$(4.6.12)$}
\label{4612}
\end{center}
\begin{center}		
\includegraphics[width=30mm]{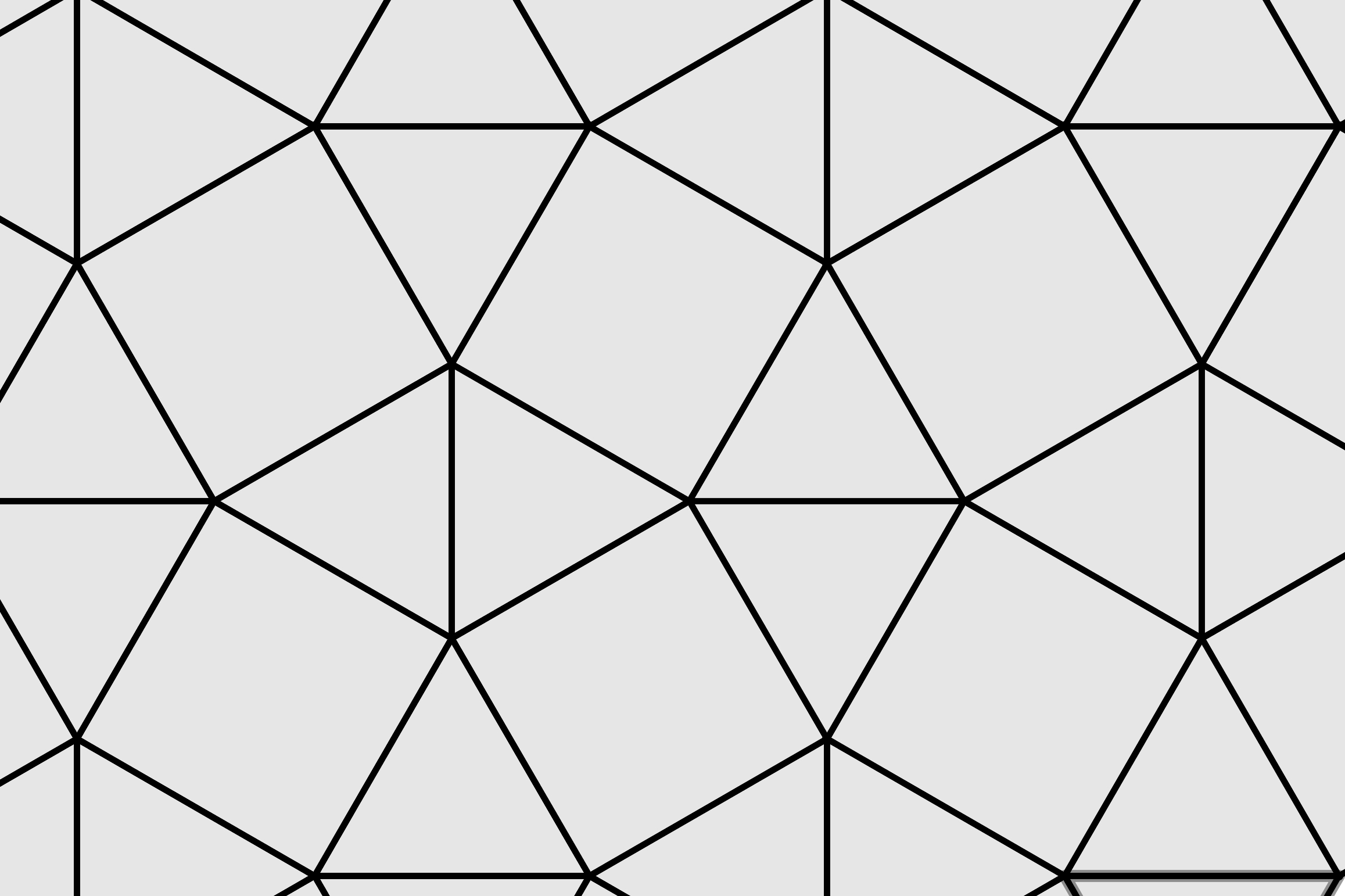}
\caption{$(3^2.4.3.4)$}
\label{33434}
\end{center}
\end{multicols}
\end{figure}

\begin{figure}[h!]
\begin{multicols}{2}
\begin{center}		
\includegraphics[width=30mm]{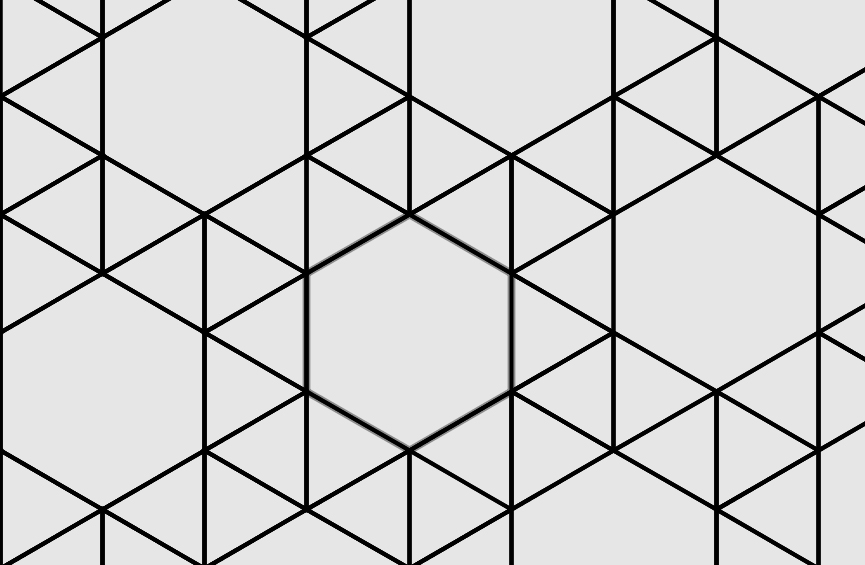}
\caption{$(3^4.6)$}
\label{33336}
\end{center}
\begin{center}		
\includegraphics[width=30mm]{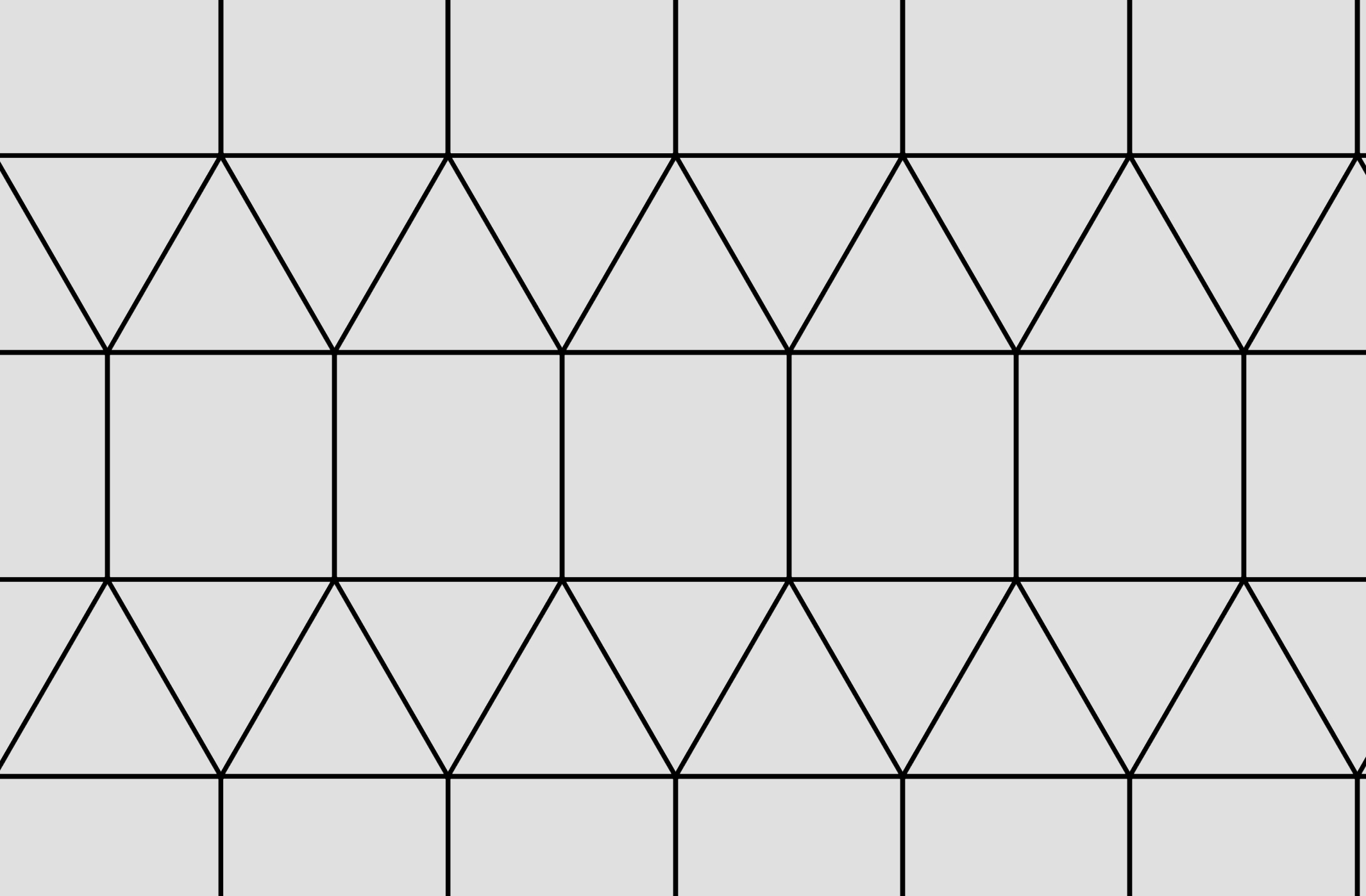}
\caption{$(3^3.4^2)$}
\label{33344}
\end{center}
\end{multicols}
\end{figure}

The equivelar Archimedean tessellations of type $\{3,6\},\,\{4,4\},\,$ and $\{6,3\}$ and the corresponding toroidal maps were considered in~\cite{DMEquivelar}. In this paper we will mainly work with the non-equivelar Archimedean tessellations of the Euclidean plane; denote $\mathcal A$ to be the set of all non-equivelar tessellations.

  Furthermore, for reference, on each such tessellation we can place a Cartesian coordinate system with the origin at a vertex of the tessellation.  For the tessellations $\{4,4\}$, $(3^2,4,3,4)$, and $(4.8^2)$, the coordinate system is further specified by assuming that the vectors ${\bf e_1}:=(1,0)$ and ${\bf e_2}:=(0,1)$ represent the shortest possible translational symmetries of the tessellation. Similarly, for the remaining Archimedean tessellations, other than $(3^3.4^2)$, we assume that the vectors ${\bf e_1}:=(1,0)$ and ${\bf e_2}:=\left(1/2, \sqrt{3}/2\right)$ represent the shortest possible translational symmetries. 

This choice of coordinate system will allow us to utilize the geometry of the Gaussian and Eisenstein integers in our results (see Subsection~\ref{Subsec:Gauss}). Given a tessellation $\tau$, we call $(\bold{e_1}, \bold{e_2})$ \textit{the basis for $\tau$}.

For the tessellation $(3^3.4^2)$ the basis will be specified separately in the later sections.

Let $\left(\bf  e_1,\bf  e_2\right)$  be the basis for $\tau \in \mathcal A \setminus\{(3^3.4^2)\}$. The set $\left\{\lambda{\bf e_1}+\mu{\bf e_2} \colon \lambda,\mu \in \mathbb{Z}\right\}$ forms the vertex set of a regular tessellation which we call \emph{the tessellation associated with $\tau$}, and which we denote by $ \tau^*$. 
 By construction, $\left(\bf  e_1,\bf  e_2\right)$ is also the basis for $ \tau^*$. 

Given an Archimedean tessellation $\tau$, denote by $T_\tau$ the maximal group of translations that preserve $\tau$. As we already mentioned, Archimedean toroidal maps can be seen as quotients of planar Archimedean tessellations. These quotients can be written explicitly in terms of possible subgroups of $T_\tau$, due to the following theorem.

\begin{theorem}[Archimedean toroidal maps are quotients; \cite{Neg}]
\label{intprop1}	
Let $\mathcal M$ be an Archimedean map on the torus. Then there exists an Archimedean tessellation $\tau$ of the Euclidean plane and a subgroup $G \leq T_\tau$ so that $\mathcal M = \tau / G$. \qed
\end{theorem}

This theorem shows that, for each type, there is a one-to-one correspondence between Archimedean toroidal maps and translation subgroups of $T_\tau$; clearly, the pair of generators of every such subgroup should be comprised of non-collinear vectors.

 We point out here that the converse of Theorem~\ref{intprop1} is also true; any map on the torus that is obtained as a quotient of an Archimedean tessellation by a translation subgroup is an Archimedean toroidal map.   

Let $\mathcal M$ be an Archimedean toroidal map; by Theorem~\ref{intprop1}, it can be written as $\tau/\left<\bold a, \bold b\right>$, where $\tau$ is the planar Archimedean tessellation (of the same type as $\mathcal M$) and $\left<\bold a, \bold b\right> \leq T_\tau$ is the translation subgroup with generators $\bold a, \bold b \in T_\tau$. We use the standard notation $\tau_{\bold a, \bold b} := \tau / \left<\bold a, \bold b\right>$ for the map $\mathcal M$. Note that the pair $\bold a, \bold b$ is not uniquely defined by $\mathcal M$, but the quotient is independent of possible choices.

A \textit{flag} of a planar tessellation $\tau$ is a triple of an incident vertex, edge, and face of the tessellation.  We can then define a flag of a toroidal map $\tau_{\bold a, \bold b}$ as the orbit of a flag under the group $\left<\bold a, \bold b\right>$.  We note that when the map is combinatorially equivalent to an {abstract polytope} (see~\cite{McMull}), this is equivalent to a flag equaling a triple of an incident vertex, edge, and face of the map itself. Two flags of a map on the torus are said to be \textit{adjacent} if they lift to flags in the plane that differ in exactly one element.

Let $\calN= \tau / H $ and $\calM=  \tau / G$ be Archimedean maps on the torus, where $H$ is a subgroup of $G$.   Then there is a surjective function $\eta\colon  \calN \to \calM $ that preserves adjacency and sends vertices of $\calN$ to vertices of $\calM$ (and edges to edges, and faces to faces).  The function $\eta$   is called a \emph{covering} of the map $\mathcal M$ by the map $\mathcal N$. This is denoted by $\mathcal N \searrow \mathcal M$, and we say that $\mathcal N$ is a \emph{cover} of $\mathcal M$.  
We can use the notion of covering to create a partial order $\leq$ on a set $\mathcal S$ of toroidal maps, where $\calM \leq \calN$ if and only if $\calN $ is a cover of $\calM$.  A {\em minimal cover} of a map $\calM$ is minimal with respect to this partial order in a set of all maps that cover $\calM$.   We note here that this notion of covering can also be generalized to maps on different surfaces, and to abstract polytopes of higher ranks.  If $\mathcal S$ is the set of all regular maps that cover a given map $\calM$, then the minimal elements of the partial order are the minimal regular covers of $\calM$, as studied in~\cite{Hartley_Min_cov,Tomotope} for example.

For a map $\mathcal M = \tau / G$, where $G = \left<\bold a, \bold b\right>$, we call the parallelogram spanned by the vectors $\bold a$, $\bold b$ a \textit{fundamental region} of $\mathcal M$. Then, a fundamental region for the covering map $\mathcal N = \tau / H$, $H \leq G$, can be viewed as $K$ fundamental regions of $\mathcal M$ `glued together' (see Figure~\ref{exgluing}). It is easy to show that the number $K$ is equal to the index $[G:H]$ of the subgroup $H$ in $G$.

\begin{figure}[h]
\begin{center}
\begin{small}
\definecolor{cqcqcq}{rgb}{0.75,0.75,0.75}
\definecolor{zzzzzz}{rgb}{0.6,0.6,0.6}
\begin{tikzpicture}[line cap=round,line join=round,x=1.0cm,y=1.0cm,scale = 0.4]
\draw [color=cqcqcq,, xstep=1.0cm,ystep=1.0cm] (-5,-1) grid (12,10);
\fill[color=zzzzzz,fill=zzzzzz,fill opacity=0.1] (0,0) -- (1,3) -- (5,3) -- (4,0) -- cycle;
\clip(-5,-1) rectangle (12,10);
\draw [-{Stealth[length=2mm]}] (0,0) -- (4,0);
\draw [-{Stealth[length=2mm]}] (0,0) -- (1,3);
\draw [dash pattern=on 2pt off 2pt] (-4,0)-- (0,0);
\draw [dash pattern=on 2pt off 2pt] (-4,0)-- (-3,3);
\draw [dash pattern=on 2pt off 2pt] (-3,3)-- (1,3);
\draw [dash pattern=on 2pt off 2pt] (1,3)-- (5,3);
\draw [dash pattern=on 2pt off 2pt] (5,3)-- (4,0);
\draw [dash pattern=on 2pt off 2pt] (4,0)-- (8,0);
\draw [dash pattern=on 2pt off 2pt] (8,0)-- (9,3);
\draw [dash pattern=on 2pt off 2pt] (9,3)-- (5,3);
\draw [dash pattern=on 2pt off 2pt] (-3,3)-- (-2,6);
\draw [dash pattern=on 2pt off 2pt] (-2,6)-- (2,6);
\draw [dash pattern=on 2pt off 2pt] (1,3)-- (2,6);
\draw [dash pattern=on 2pt off 2pt] (2,6)-- (6,6);
\draw [dash pattern=on 2pt off 2pt] (6,6)-- (5,3);
\draw [dash pattern=on 2pt off 2pt] (6,6)-- (10,6);
\draw [dash pattern=on 2pt off 2pt] (10,6)-- (9,3);
\draw [dash pattern=on 2pt off 2pt] (-2,6)-- (-1,9);
\draw [dash pattern=on 2pt off 2pt] (-1,9)-- (3,9);
\draw [dash pattern=on 2pt off 2pt] (3,9)-- (2,6);
\draw [dash pattern=on 2pt off 2pt] (3,9)-- (7,9);
\draw [dash pattern=on 2pt off 2pt] (7,9)-- (6,6);
\draw [dash pattern=on 2pt off 2pt] (7,9)-- (11,9);
\draw [dash pattern=on 2pt off 2pt] (11,9)-- (10,6);
\draw [-{Stealth[length=2mm]}] (0,0) -- (9,3);
\draw [-{Stealth[length=2mm]}] (0,0) -- (-2,6);
\draw [-{Stealth[length=2mm]}] (9,3) -- (7,9);
\draw [-{Stealth[length=2mm]}] (-2,6) -- (7,9);
\draw (3.21,0.0) node[anchor=north west] {$\bold{a}$};
\draw (1.5,2.5) node {$\bold b$};
\draw (9.5,2.5) node {$\bold u$};
\draw (-2.5,6) node {$\bold v$};
\draw (2.5,1.46) node {1};
\draw (3.5,4.46) node {2};
\draw (4.5,7.46) node {3};
\draw (7,7) node {4};
\draw (7,4.5) node {5};
\draw (0,2) node {5};
\draw (5.5,2.5) node {3};
\draw (0,4.5) node {4};
\draw (1.5,6.5) node {1};
\fill [color=black] (0,0) circle (2.5pt);
\fill [color=black] (4,0) circle (2.5pt);
\fill [color=black] (1,3) circle (2.5pt);
\fill [color=black] (5,3) circle (2.5pt);
\fill [color=black] (8,0) circle (2.5pt);
\fill [color=black] (-4,0) circle (2.5pt);
\fill [color=black] (2,6) circle (2.5pt);
\fill [color=black] (6,6) circle (2.5pt);
\fill [color=black] (9,3) circle (2.5pt);
\fill [color=black] (10,6) circle (2.5pt);
\fill [color=black] (-3,3) circle (2.5pt);
\fill [color=black] (-2,6) circle (2.5pt);
\fill [color=black] (-1,9) circle (2.5pt);
\fill [color=black] (3,9) circle (2.5pt);
\fill [color=black] (7,9) circle (2.5pt);
\fill [color=black] (11,9) circle (2.5pt);
\end{tikzpicture}
\end{small}

\caption{$\{4,4\}_{\bold u, \bold v} \searrow \{4,4\}_{\bold a, \bold b}$ is a $5$-sheeted covering, and the covering map $\{4,4\}_{\bold u, \bold v}$ is obtained by gluing together $5$ fundamental regions of $\{4,4\}_{\bold a, \bold b}$.} 
\label{exgluing}
\end{center}

\end{figure}
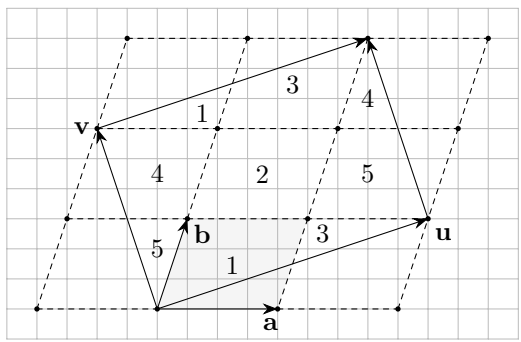

\subsection{Symmetries and automorphisms of tessellations and maps}

In this section we follow~\cite{ToroidsFewOrbits, HOPW, OrientableTwistoids} and~\cite{Didicosm} in our notation and definitions.

Let $\tau$ be a tessellation of the Euclidean plane, and let $\Aut(\tau)$ be its symmetry group (the collection of isometries of the Euclidean plane that preserve the tessellation). 
Let $G$ and $H$ be two subgroups of $\Aut(\tau)$ generated by two linearly independent translations.  The maps $\tau /G$ and $\tau/H$ are isomorphic if $G$ and $H$ are conjugate in $\Aut(\tau)$.

A symmetry $\gamma \in \Aut(\tau)$ induces an automorphism of a toroidal map $\tau /G$ if and only if it normalizes $G$, that is $\gamma G \gamma^{-1} = G$; denote $\Norm_{\Aut(\tau)}(G)$ for the group of elements in $\Aut(\tau)$ that normalize $G$. Geometrically, such $\gamma$ maps fundamental regions of $\tau /G$ to  fundamental regions of $\tau /G$. 

Finally, the we define the automorphism group $\Aut(\tau/G)$ as the group induced by the normalizer $\Norm_{\Aut(\tau)}(G)$; in other words $\Aut(\tau/G) \cong \Norm_{\Aut(\tau)}(G)/G$. We will also denote the collection of symmetries $\Norm_{\Aut(\tau)}(G)$ as simply $\Sym(\tau/G)$.
 
We note here that an automorphism of a map can equivalently be defined as an automorphism of the underlying graph that can be extended to a homeomorphism of the surface.

A map $\mathcal M$ is called \emph{regular} if its automorphism group acts transitively on the set of flags. A map $\mathcal M$ is called \emph{chiral} if its automorphism group has two orbits on flags with adjacent flags lying in different orbits. A map $\mathcal M$ is called \emph{rotary} if it is either regular or chiral. 

For the toroidal maps of type $\{4,4\}$, $\{3,6\}$, and $\{6,3\}$ the minimum possible number of flag orbits is one, given by regular maps of those types (which have been previously classified, see also Subsection~\ref{Subsec:RegChiral}). For other types of maps on the torus, the minimum number of flag orbits will not be one.  However, we still would like to understand the maps of each type that achieve the fewest possible number of flag orbits. 

An Archimedean map on the torus is called \emph{almost regular} if it has the same number of flag orbits under the action of its automorphism group as the Archimedean tessellation on the plane of the same type has under the action of its symmetry group.

\subsection{Regular and chiral toroidal maps}
\label{Subsec:RegChiral}

The classification in the next sections depends heavily on the classification of regular and chiral toroidal maps.  Here we summarize the relevant details about toroidal maps of type $\{4,4\}$ and $\{3,6\}$ that are needed in our results.  The results in this subsection, and much more can all be found in~\cite{HOPW}.

Let $\tau$ be a tessellation of the Euclidean plane of type $\{4,4\}$ or $\{3,6\}$.  Then $\Aut(\tau)$ is of the form $ T_\tau \rtimes {\bf S}$, where ${\bf S}$ is the stabilizer of a vertex of $\tau$, which we can assume to be the origin without loss; ${\bf S}$ is called \emph{a point stabilizer}.    

Then let $\calM = \tau/G$ be a toroidal map.  Notice that every translation in $T_\tau$ induces an automorphism of $\calM$ (where the elements of $G$ induce the trivial automorphism).  Define $\chi$ as the central inversion of the Euclidean plane, that is the symmetry that sends any vector ${\bf u}$ to ${\bf  - u}$.  Then $\Aut(\calM)$ is induced by a group ${\bf K}$ so that $ T_\tau \rtimes \langle \chi \rangle \leq {\bf K} \leq \Aut(\tau)$. 

Furthermore, there is a bijection between such groups ${\bf K}$ and subgroups ${\bf K'}$ of ${\bf S}$ containing $\chi$, and thus one needs to determine which symmetries in the point stabilizer ${\bf S}$ normalize $G$.   Finally, the number of flag orbits of the toroidal map $\calM$ is the index of $\Norm_{\Aut(\tau)}(G)$ in $\Aut(\tau)$, which is the same as the index of $ {\bf K'}$ in $ {\bf S}$.

First let us consider toroidal maps of type $\{4,4\}$; let $\tau$ be the regular tessellation of the Euclidean plane of this type, and $({\bf e_1}, {\bf e_2})$ be the basis for $\tau$.  The point stabilizer  $ {\bf S}$ is generated by two reflections $R_1$ and $R_2$, where $R_1$ is reflection across the line spanned by ${\bf e_1 + e_2}$, sending vectors $(x,y)$ to $(y,x)$, and $R_2$ is reflection across the line spanned by ${\bf  e_1}$, sending vectors $(x,y)$ to $(x,-y)$.  There are exactly three conjugacy classes of proper subgroups $ {\bf K'}$ of $ {\bf S}$ containing $\chi$ but not equal to $\langle \chi \rangle$. In other words there are exactly five possible point stabilizers: all of  ${\bf S}$, only $\langle \chi \rangle$, and finally the three groups described next. The three subgroups are $ {\bf K'}$ are $\langle \chi, R_1 \rangle $, $\langle \chi, R_2 \rangle $, and $\langle \chi, R_1 R_2 \rangle $, and each has index 2 in ${\bf S}$, where  $\langle \chi \rangle$ has index 4 in ${\bf S}$.  

  We note that it is important for our classification to notice that a toroidal map of type $\{4,4\}$ is regular if and only if $ {\bf K'}$ contains both $R_1$ and $R_2$, as well as $R_1R_2$ which is the rotation by ${\pi}/{2}$ around the origin.  This occurs only in the two well known families of regular toroidal maps, $\{4,4\}_{(a,0),(0,a)}$ and $\{4,4\}_{(a,a),(a,-a)}$, both of which have squares as fundamental regions. The chiral toroidal maps of type $\{4,4\}$ also have squares as their fundamental regions, but have no reflections in their automorphism groups. Finally, the remaining classes of toroidal maps have fundamental regions that are not squares.

Next, let us consider toroidal maps of type $\{3,6\}$; let $\tau$ be the regular tessellation of the Euclidean plane of this type.   We use the previously described basis of ${\bf e_1}=(1,0)$ and ${\bf e_2}=\left(1/2, \sqrt{3}/2\right)$ to describe the symmetries of these maps.

 The point stabilizer  $ {\bf S}$ is again generated by two reflections $R_1$ and $R_2$, where $R_1$ is reflection across the line spanned by ${\bf e_1 + e_2}$, sending vectors $(x,y)$ to $(y,x)$, and $R_2$ is reflection across the line spanned by ${\bf  e_1}$, sending vectors $(x,y)$ to $(x+y,-y)$.

 Now there are exactly two conjugacy classes of proper subgroups $ {\bf K'}$ of $ {\bf S}$ containing $\chi$ but not equal to $\langle \chi \rangle$. These subgroups $ {\bf K'}$ are $\langle \chi, R_1R_2 \rangle $, with index 2 in ${\bf S}$, and $\langle \chi, R_2 \rangle $ with index 3 in ${\bf S}$, where  $\langle \chi \rangle$ has index 6 in ${\bf S}$.  

  We note that it is important for our classification to notice that a toroidal map of type $\{3, 6\}$ is regular if and only if $ {\bf K'}$ contains both $R_1$ and $R_2$, as well as $R_1R_2$ which is the rotation by ${\pi}/{3}$ around the origin.  This occurs only in the two well known families of regular toroidal maps, $\{3,6\}_{(a,0),(0,a)}$ and $\{3,6\}_{(a,a),(2a,-a)}$. For those two families the fundamental regions are parallelograms composed of two regular triangles. The chiral toroidal maps of type $\{3,6\}$ also have parallelograms composed of two regular triangles as their fundamental regions, but have no reflections in their automorphism groups; this is similar to the type $\{4,4\}$.    

\section{Almost regular maps}
\label{Sec:AlmostReg}

In this section we consider Archimedean tessellations of the torus with as much symmetry as possible.  As we already mentioned, one natural way to understand the symmetry of a map is to consider the action of its automorphism group on its flags.  Here we want to understand the maps on the torus with as few flag orbits as possible.  

\begin{theorem}[Regular to almost regular maps]
\label{Thm:RegAlmostRegular}
Let 
$$
\mathcal{A}_{\text{reg}} := \left\{(4.8^2),\,(3.6.3.6),\,(3.12^2),\,(4.6.12),\,(3.4.6.4),\,(3^2.4.3.4)\right\}
$$ 
and $\tau \in \mathcal A_{\text{reg}}$ be an Archimedean tessellation of one of these types. Then $\tau_{\bold{u},\bold{v}}$ is an almost regular Archimedean map if and only if $ \tau^*_{\bold{u}, \bold{v}}$ is a regular map on the torus.
\end{theorem}

The proof of this theorem will follow from the following six propositions, each separately dealing with a type of map.

\begin{proposition}[Almost regular maps of type $(4.8^2)$]
\label{Prop:4.8^2}
A map $\calM = \tau/G$ of the torus of type $(4.8^2)$ is almost regular (with three flag orbits) if and only if $\tau^*/G$ is regular.
\end{proposition}

\begin{proof}

Notice first that $\tau^*$ is of type $\{4,4\}$, and let $({\bf e_1},{\bf e_2})$ be the basis for $\tau$ (and hence for $\tau^*$).

Assume that a map $\calM$ of the torus of type $(4.8^2)$ has exactly three flag orbits. For this to be the case, there must be the following symmetries in $\Sym(\calM)$, as shown in Figure~\ref{f488}:
\begin{itemize}
\item 
reflection across a line in the direction ${\bf  e_1 + e_2}$ through the center of a square of the map;
\item 
reflection across a line in the direction ${\bf  e_1 }$ through the center of a square and the edge of an octagon.
 \end{itemize}

It was summarized in Subsection~\ref{Subsec:RegChiral} that the existence of these symmetries is enough to show that  $\tau^*/G$ is regular.

Conversely, if $\tau^*/G$ is regular, then the fundamental region of $\calM$ is a square, and each of the previous listed symmetries are elements of $\Sym(\calM)$. Furthermore, every translational symmetry in $\Sym(\tau/T_\tau)$ is also in $\Sym(\tau/G)$, and thus $\calM$ has only three flag orbits. Note that the translations in $\Sym(\tau/T_\tau)$ act on the flags in 24 flag orbits, and then the listed symmetries force there to only be three orbits.

\begin{figure}[h]
$$\includegraphics[width=50mm]{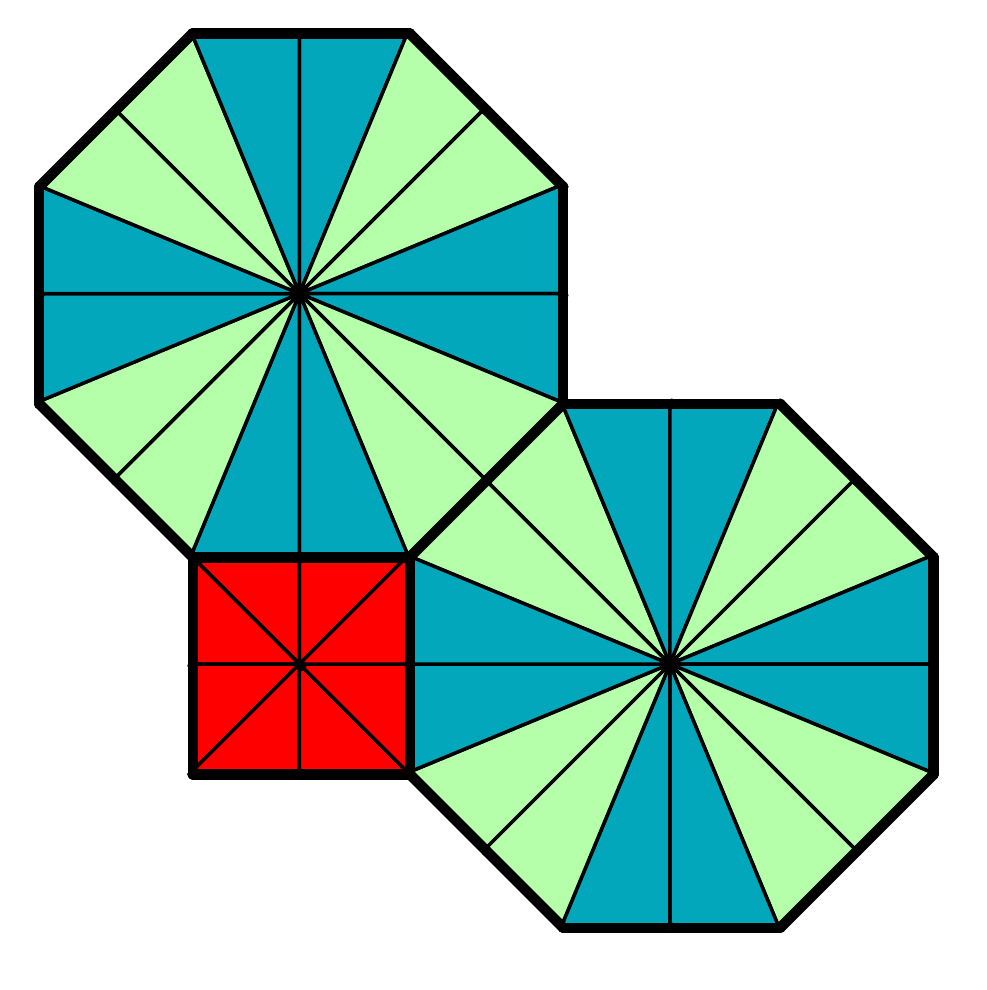} \hspace{1cm} \includegraphics[width=50mm]{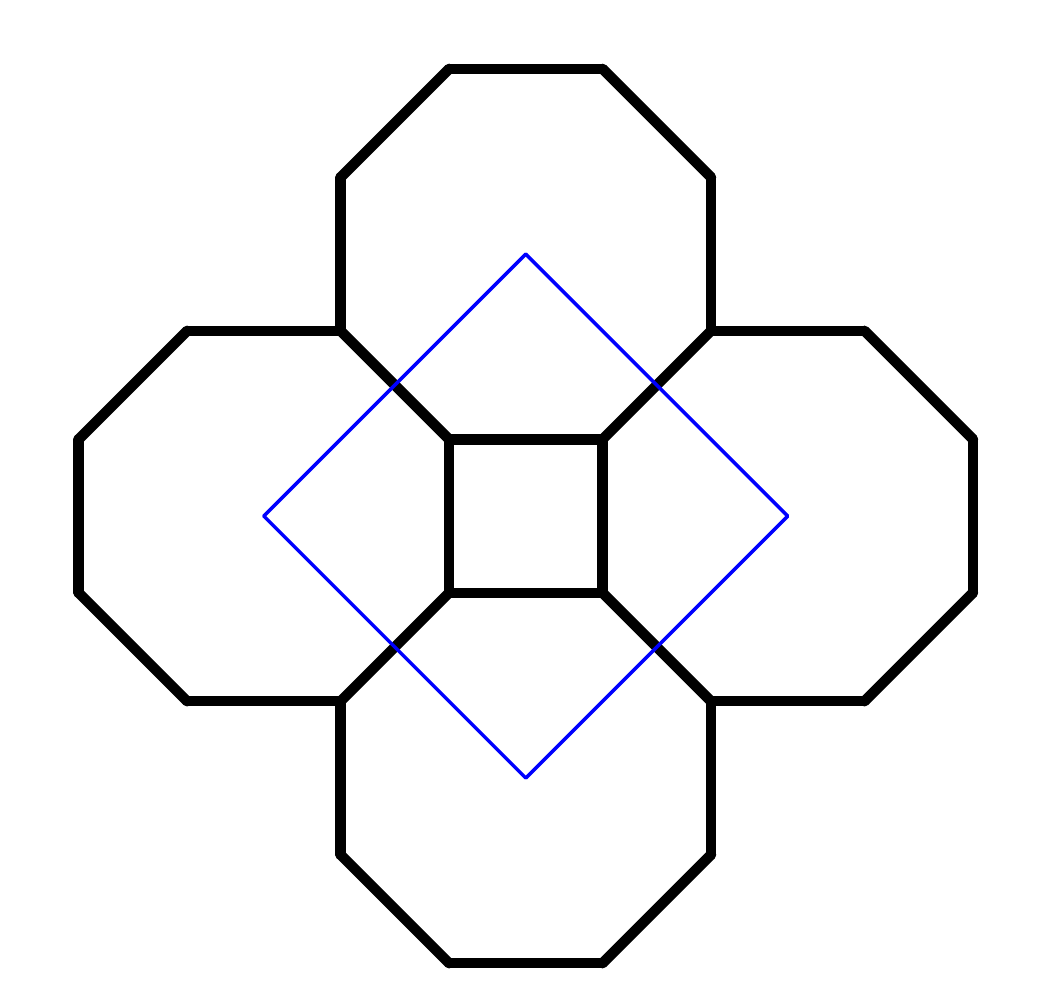}$$
\caption{Minimum number of flag orbits for $\tau$ of type $(4.8^2)$, and a fundamental region of $\tau/T_\tau$.}
\label{f488}
\end{figure}

On the left of Figure~\ref{f488}, and other figures to follow, all the faces incident to single vertex are shown, and the flags in these faces, which can be seen as triangles in the barycentric subdivision of the tessellation, are colored based on their orbit.  On the right of Figure~\ref{f488}, the fundamental region of  $\tau/T_\tau$ is shown in blue, with the underlying tessellation shown in black.  One can see, for example, that this fundamental region contains 24 flags of $\tau$.
\end{proof}

The proofs of the following five proposition is similar to the proof of Proposition~\ref{Prop:4.8^2}.

\begin{proposition}[Almost regular maps of type $(3.6.3.6)$]
A map $\calM = \tau/G$ of the torus of type $(3.6.3.6)$ is almost regular (with two flag orbits) if and only if $\tau^*/G$ is regular.
\end{proposition}

\begin{proof}
Notice first that $\tau^*$ is of type $\{3,6\}$, and let $({\bf e_1},{\bf e_2})$ be the basis for $\tau$.

Assume that a map $\calM$ of the torus of type $(3.6.3.6)$ has exactly two flag orbits.   For this to be the case, there must be the following symmetries in $\Sym(\calM)$, as shown in Figure~\ref{f3636}:
\begin{itemize}
\item 
reflection across a line in the direction ${\bf  e_1 + e_2}$ going through the centers of a hexagon and an adjacent triangle;
\item 
reflection across a line in the direction ${\bf  e_1}$ going through the centers of two hexagon incident to the same vertex.

\end{itemize}
It was summarized in Subsection~\ref{Subsec:RegChiral} that the existence of these symmetries is enough to conclude that  $\tau^*/G$ is regular.

\begin{figure}[h]
$$\includegraphics[width=50mm, angle =30]{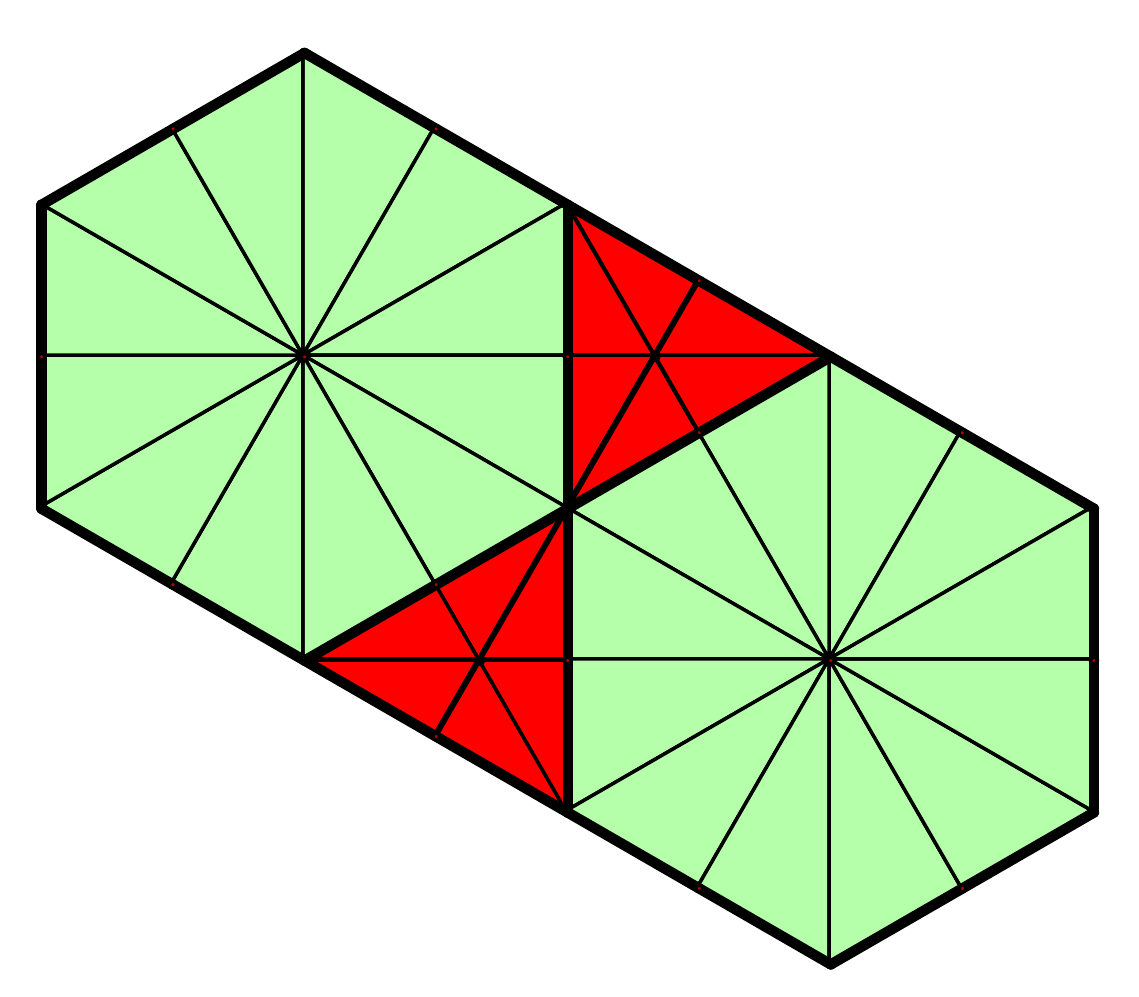} \hspace{1cm} \includegraphics[width=50mm, angle =90]{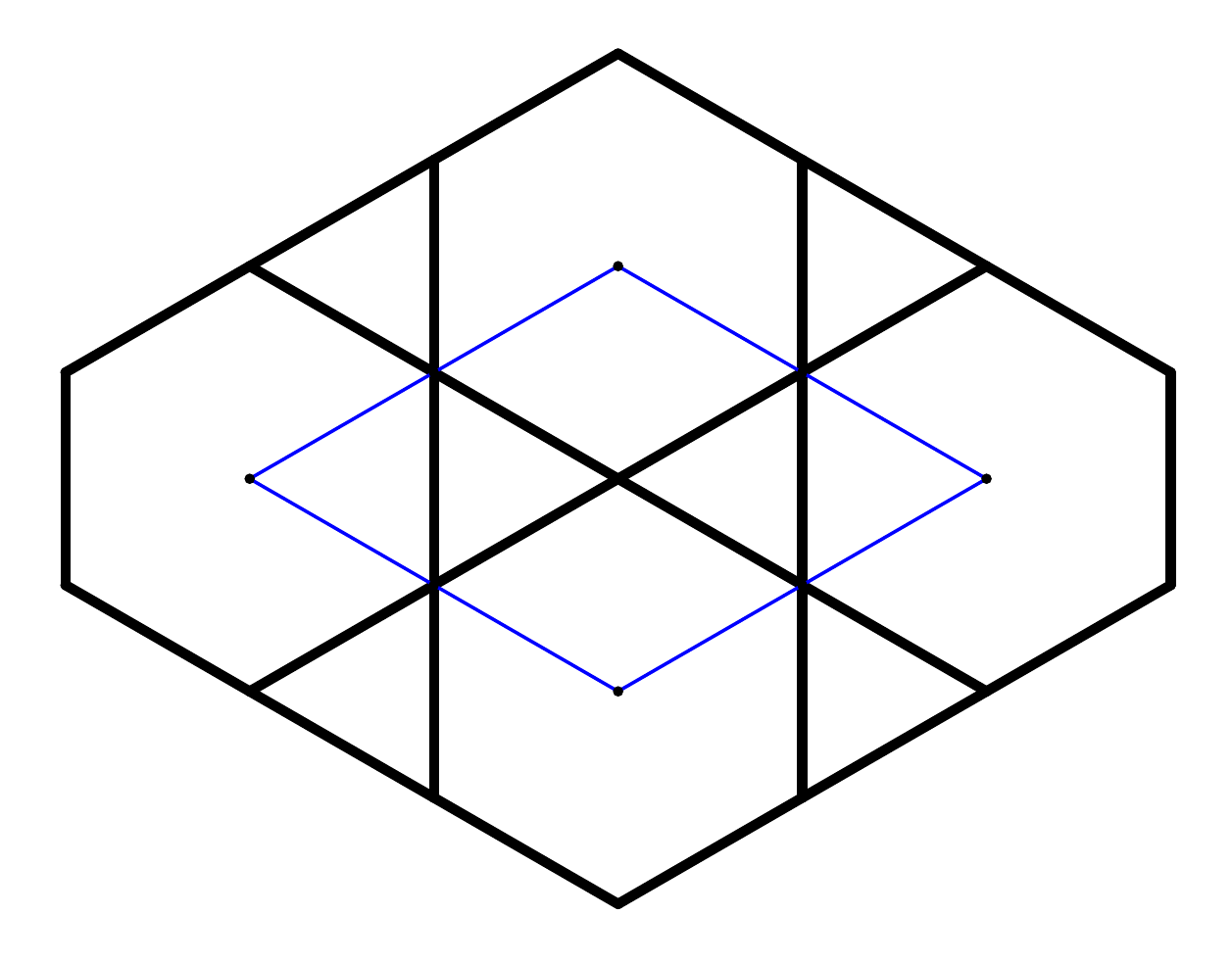}$$
\caption{Minimum number of flag orbits for $\tau$ of type $(3.6.3.6)$, and a fundamental region of $\tau/ T_\tau$.}
\label{f3636}
\end{figure}

Conversely, if $\tau^*/G$ is regular, then each of the previous listed symmetries are elements of $\Sym(\calM)$.  Furthermore, every translational symmetry in $\Sym(\tau/ T_\tau)$ is also in $\Sym(\tau/G)$, and thus $\calM$ has only two flag orbits. Note that the translations in $\Sym(\tau/ T_\tau)$ act on the flags in 24 flag orbits, and then the listed symmetries force there to only be two orbits.
\end{proof}

\begin{proposition}[Almost regular maps of type $(3.12^2)$]
A map $\calM = \tau/G$ of the torus of type $(3.12^2)$ is almost regular (with three flag orbits) if and only if $\tau^*/G$ is regular.
\end{proposition}

\begin{proof}

Notice again that $\tau^*$ is of type $\{3,6\}$, and let $({\bf e_1},{\bf e_2})$ be the basis for $\tau$.

Assume that a map $\calM$ of the torus of type $(3.12^2)$ has exactly three flag orbits. For this to be the case, there must be
the following symmetries in $\Sym(\calM)$, as shown in Figure~\ref{f3121212}:

\begin{itemize}
\item 
reflection across a line in the direction ${\bf  e_1 + e_2}$ going through the centers of a 12-gon and an adjacent triangle;
\item 
reflection across a line in the direction ${\bf  e_1}$ going through the centers of two adjacent 12-gons. 
 
\end{itemize}
As in the previous proposition, the existence of these symmetries again forces $\tau^*/G$ to be regular.

\begin{figure}[h]
$$\includegraphics[height=50mm, angle =90]{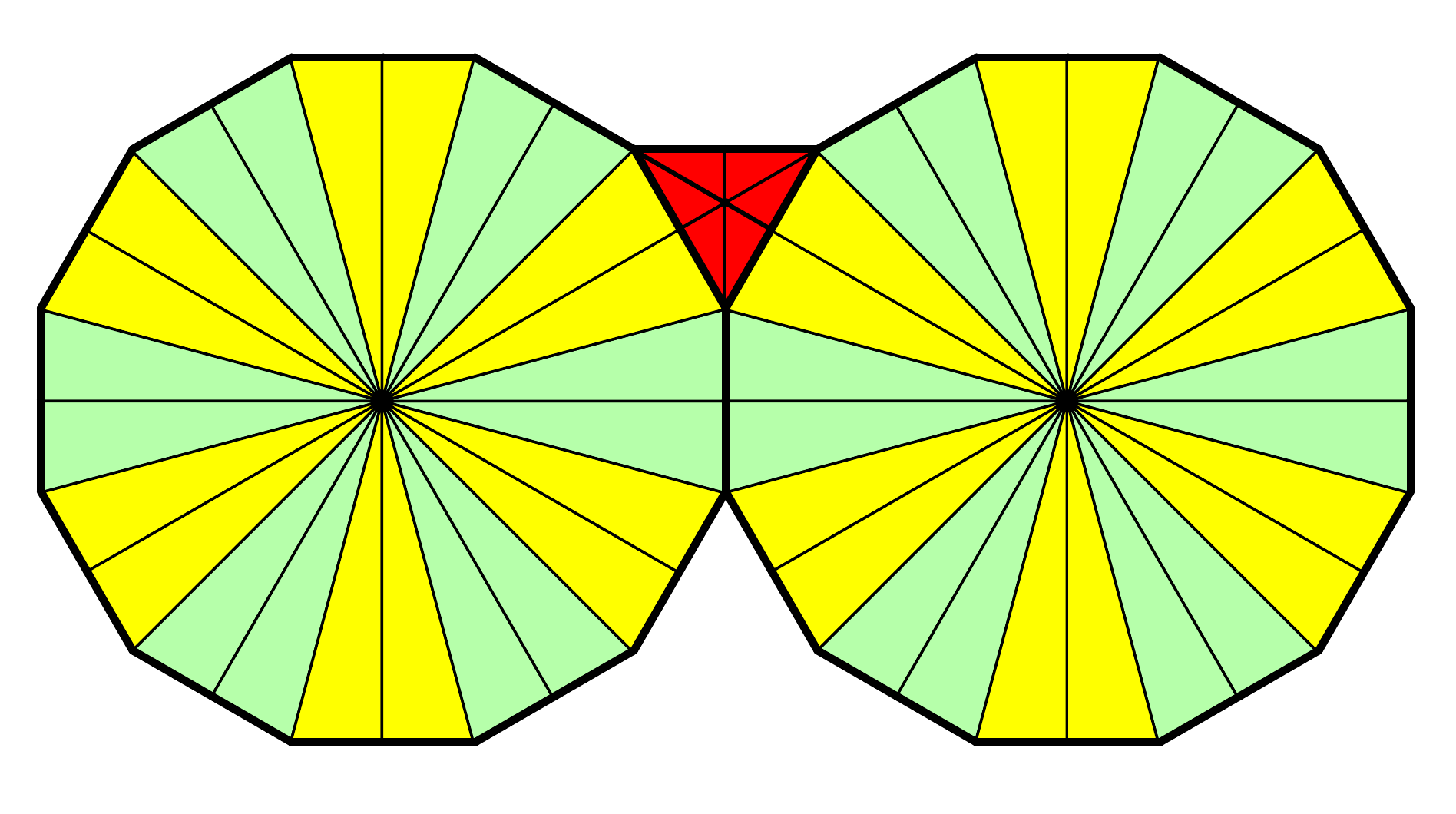} \includegraphics[width=50mm]{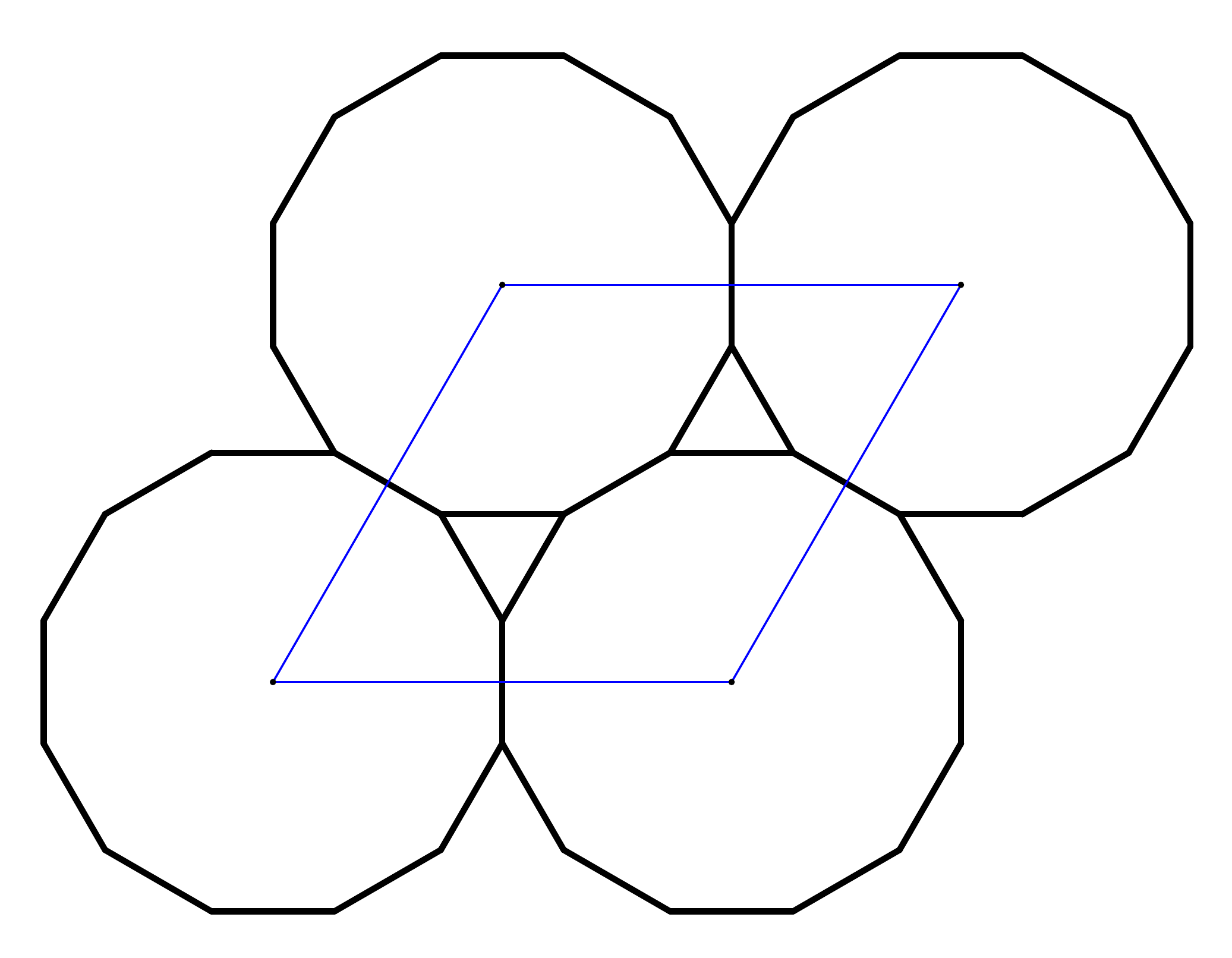}$$
\caption{Minimum number of flag orbits for $\tau$ of type  $(3.12^2)$, and a fundamental region of $\tau/H$.}
\label{f3121212}
\end{figure}

Conversely, if $\tau^*/G$ is regular, then each of the previous three listed symmetries are elements of $\Sym(\calM)$. Furthermore, every translational symmetry in $\Sym(\tau/ T_\tau)$ is also in $\Sym(\tau/G)$, and thus $\calM$ has only three flag orbits. Note that the translations in $\Sym(\tau/ T_\tau)$ act on the flags in 36 flag orbits, and then the listed symmetries force there to only be three orbits.
\end{proof}

\begin{proposition}[Almost regular maps of type $(4.6.12)$]
A map $\calM = \tau/G$ of the torus of type $(4.6.12)$ is almost regular (with six flag orbits) if and only if $\tau^*/G$ is regular.
\end{proposition}

\begin{proof}

Notice again that $\tau^*$ is of type $\{3,6\}$, and let $({\bf e_1},{\bf e_2})$ be the basis for $\tau$.

Assume that a map $\calM$ of the torus of type $(4.6.12)$ has exactly six flag orbits.   For this to be the case, there must be
the following symmetries in $\Sym(\calM)$:
\begin{itemize}
\item 
reflection across a line in the direction ${\bf  e_1 + e_2}$ going through the centers of a 12-gon and an adjacent hexagon;
\item 
reflection across a line in the direction ${\bf  e_1}$ going through the centers of a 12-gon and an adjacent square.

\end{itemize}

As in the previous proposition, the existence of these symmetries again forces $\tau^*/G$ to be regular.  

\begin{figure}[h]
$$\includegraphics[width=50mm]{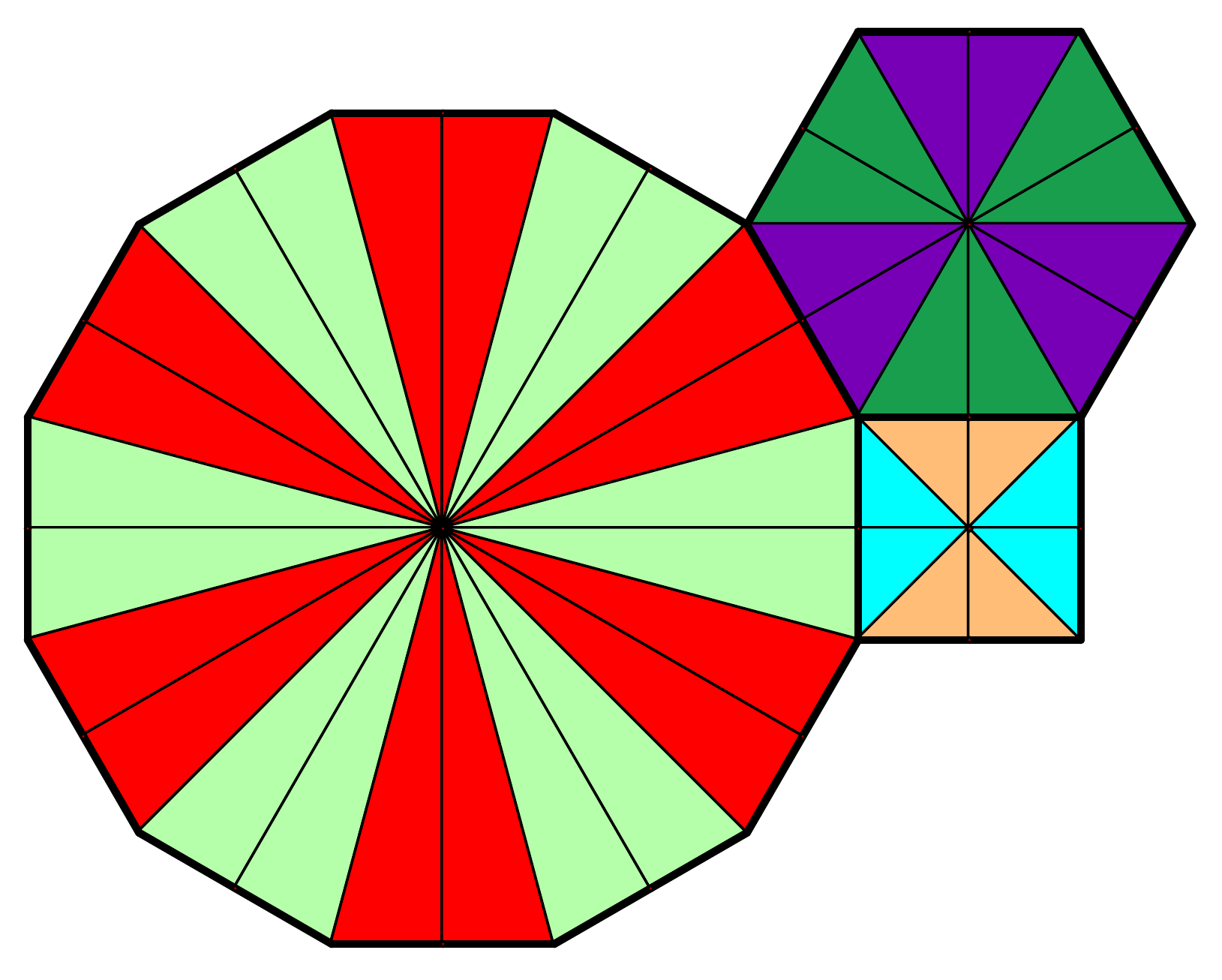} \hspace{1cm} \includegraphics[width=50mm]{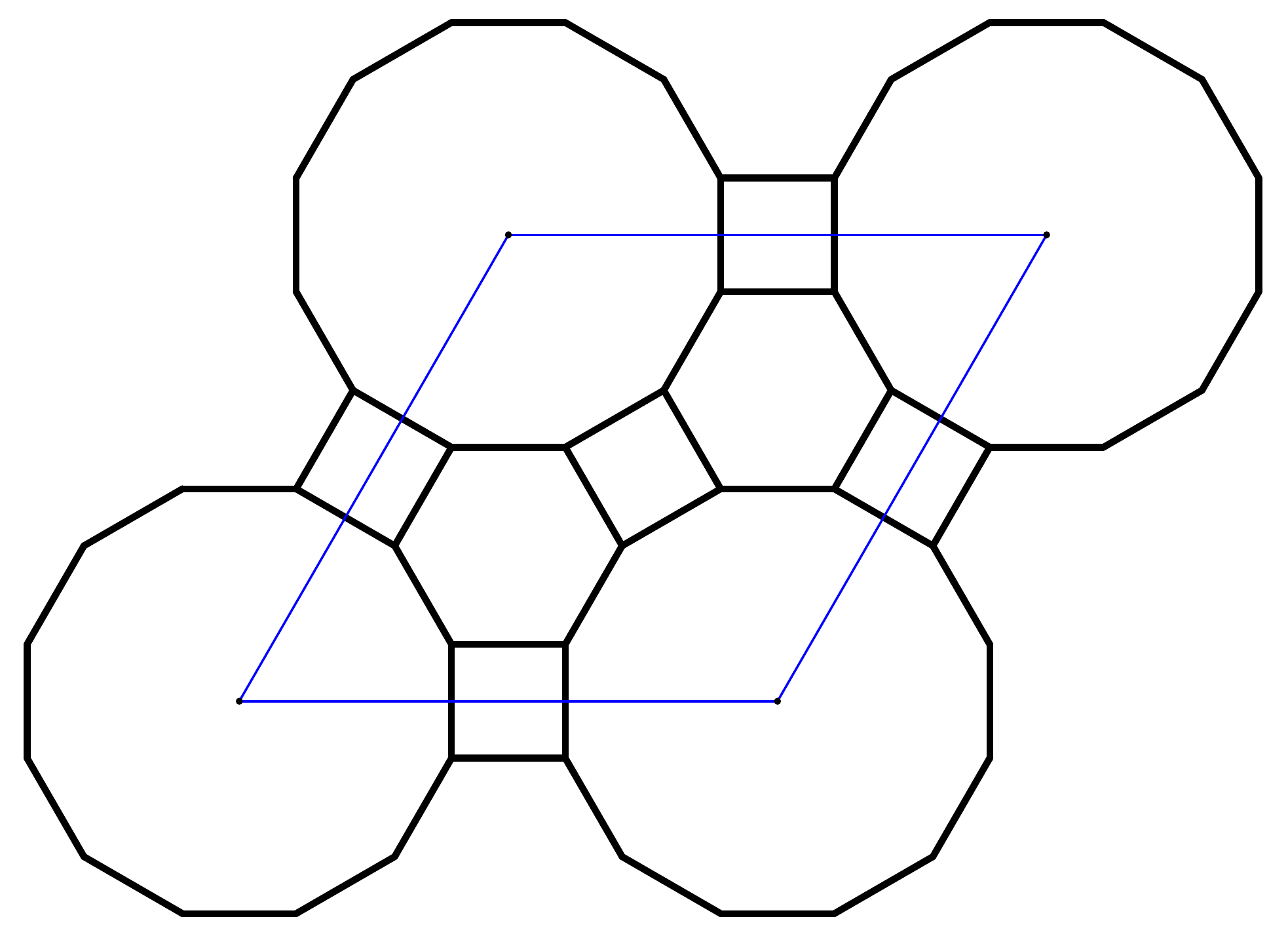}$$
\caption{Minimum number of flag orbits for $\tau$ of type  $(4.6.12)$, and a fundamental region of $\tau/ T_\tau$.}
\label{f4612}
\end{figure}

Conversely, if $\tau^*/G$ is regular, then each of the previous three listed symmetries are elements of $\Sym(\calM)$.  Furthermore, every translational symmetry in $\Sym(\tau/ T_\tau)$ is also in $\Sym(\tau/G)$, and thus $\calM$ has only six flag orbits. Note that the translations in $\Sym(\tau/ T_\tau)$ act on the flags in 72 flag orbits, and then the listed symmetries force there to only be six orbits.
\end{proof}

\begin{proposition}[Almost regular maps of type $(3.4.6.4)$]
A map $M = \tau/G$ of the torus of type $(3.4.6.4)$ is almost regular (with three flag orbits) if and only if $\tau^*/G$ is regular.
\end{proposition}

\begin{proof}

Notice again that $\tau^*$ is of type $\{3,6\}$. Let $({\bf e_1},{\bf e_2})$ be the basis for $\tau$.
Assume that a map $\calM$ of the torus of type $(3.4.6.4)$ has exactly four flag orbits.   For this to be the case, there must be
the following symmetries in $\Sym(\calM)$:
\begin{itemize}
\item 
reflection across a line in the direction ${\bf  e_1 + e_2}$ going through the centers of a hexagon and a triangle sharing an incident vertex;
\item 
reflection across a line in the direction ${\bf  e_1}$ going through the centers of a 12-gon and an adjacent square.

\end{itemize}

\begin{figure}[h]
$$\includegraphics[width=50mm]{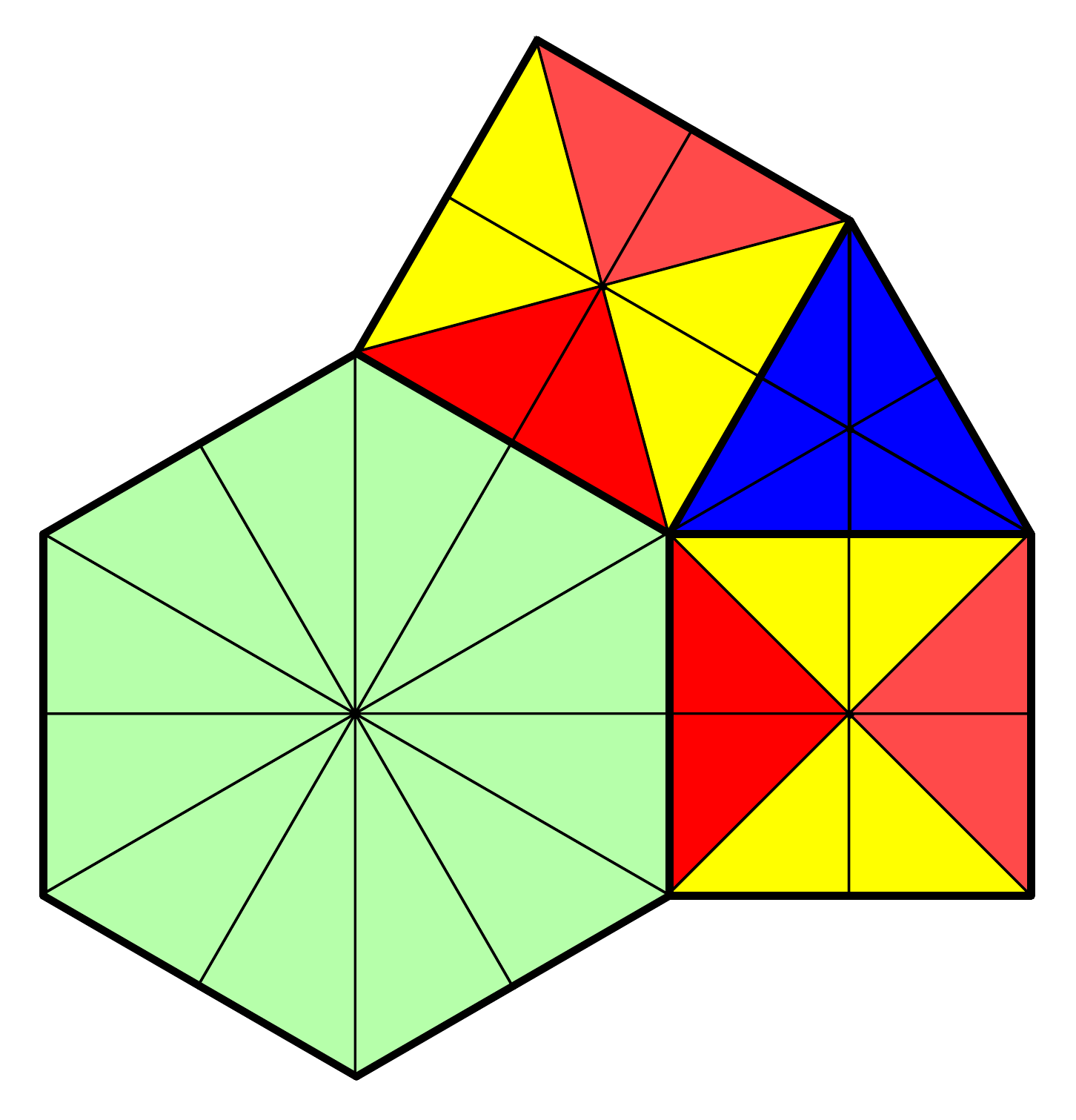} \hspace{1cm} \includegraphics[width=50mm]{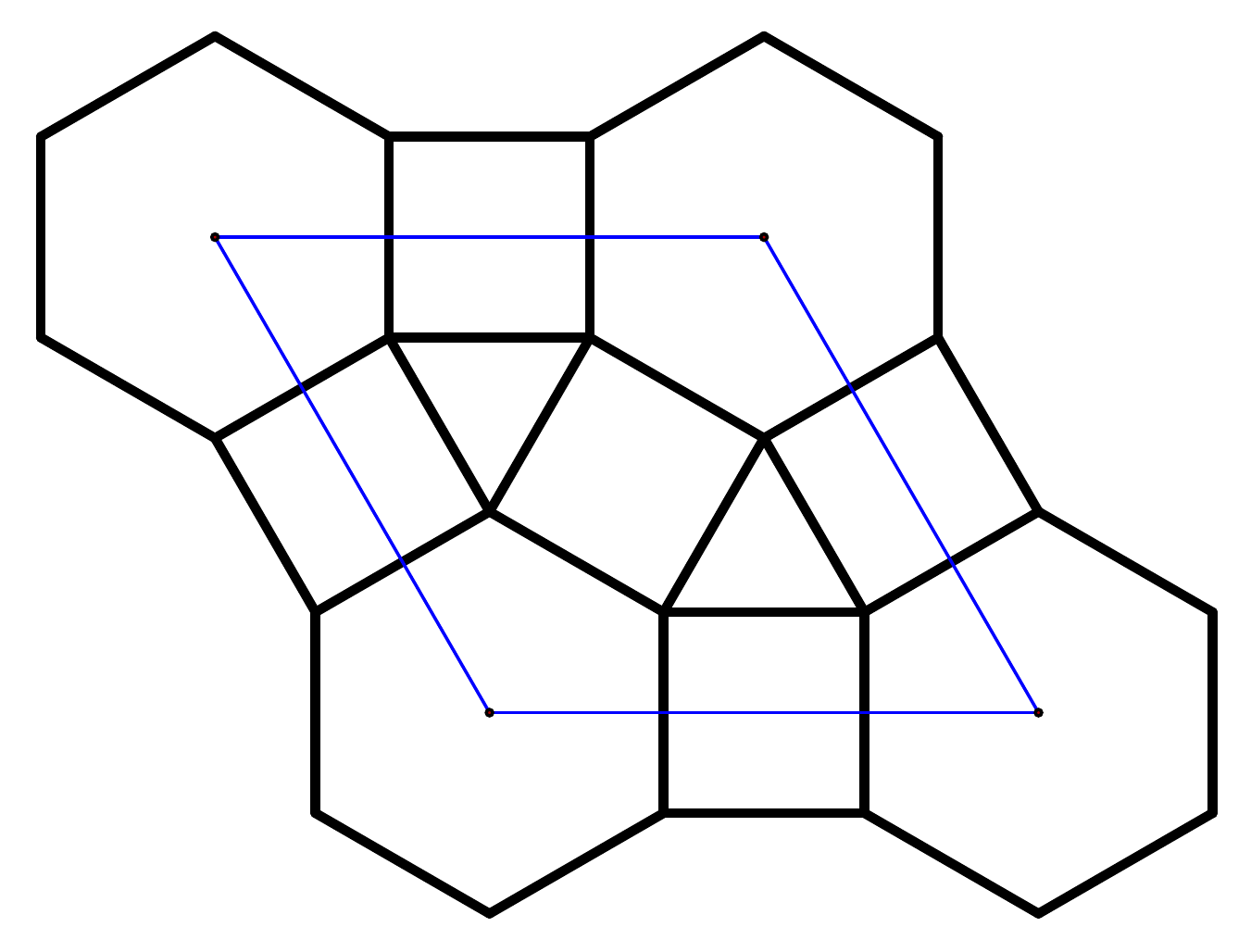}$$
\caption{Minimum number of flag orbits for $\tau$ of type  $(3.4.6.4)$, and a fundamental region of $\tau/ T_\tau$.}
\label{f3464}
\end{figure}

Again the existence of these symmetries again forces $\tau^*/G$ to be regular. Conversely, if $\tau^*/G$ is regular, then each of the previous three listed symmetries are elements of $\Sym(\calM)$.  Furthermore, every translational symmetry in $\Sym(\tau/ T_\tau)$ is also in $\Sym(\tau/G)$, and thus $\calM$ has only four flag orbits. Note that the translations in $\Sym(\tau/ T_\tau)$ act on the flags in 48 flag orbits, and then the listed symmetries force there to only be four orbits.
\end{proof}

\begin{proposition}[Almost regular maps of type $(3^2.4.3.4)$]
A map $\calM = \tau/G$ of the torus of type $(3^2.4.3.4)$ is almost regular (with five flag orbits) if and only if $\tau^*/G$ is regular.
\end{proposition}

\begin{proof}

Notice again that $\tau^*$ is of type $\{4,4\}$. Let $({\bf e_1},{\bf e_2})$ be the basis for $\tau$.

Assume that a map $\calM$ of the torus of type $(3^2.4.3.4)$ has exactly five flag orbits. For this to be the case, there must be
a rotation by ${\pi}/{2}$ around the center of a square in $\Sym(\calM)$.  This symmetry can be represented by $R_1R_2$ as described in Subsection~\ref{Subsec:RegChiral}, and thus $\tau^*/G$ is either regular or chiral.  However, $\Sym(\calM)$ must also contain a reflection across the edge of adjacent triangles in the direction of ${\bf  e_1 + e_2}$.  This means that the 8 flags in the fundamental region of $\tau^*/ T_\tau$ are all in the same orbit, and thus $\tau^*/G$ is regular.  

\begin{figure}[h]
$$\includegraphics[width=50mm]{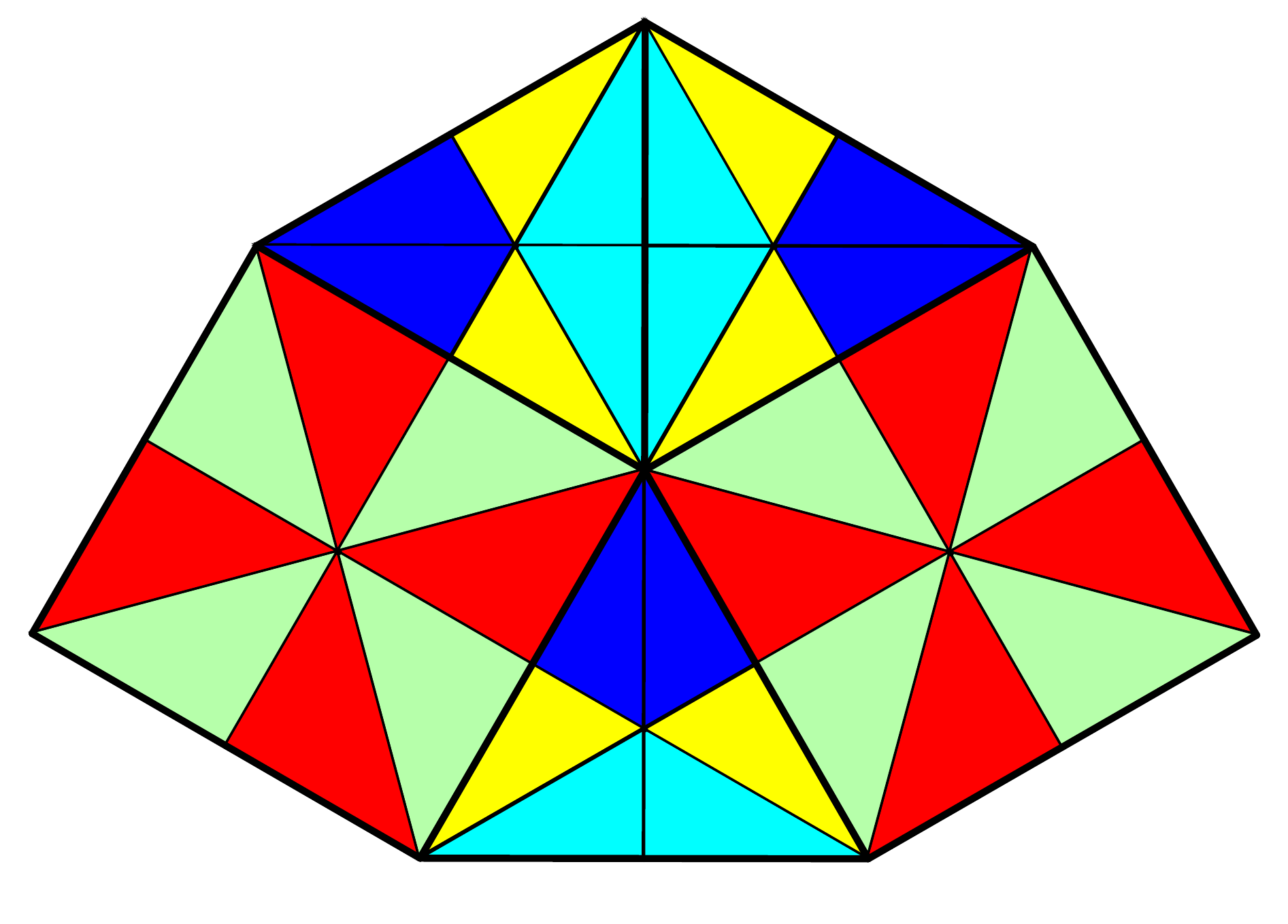} \hspace{1cm} \includegraphics[width=40mm, angle=90]{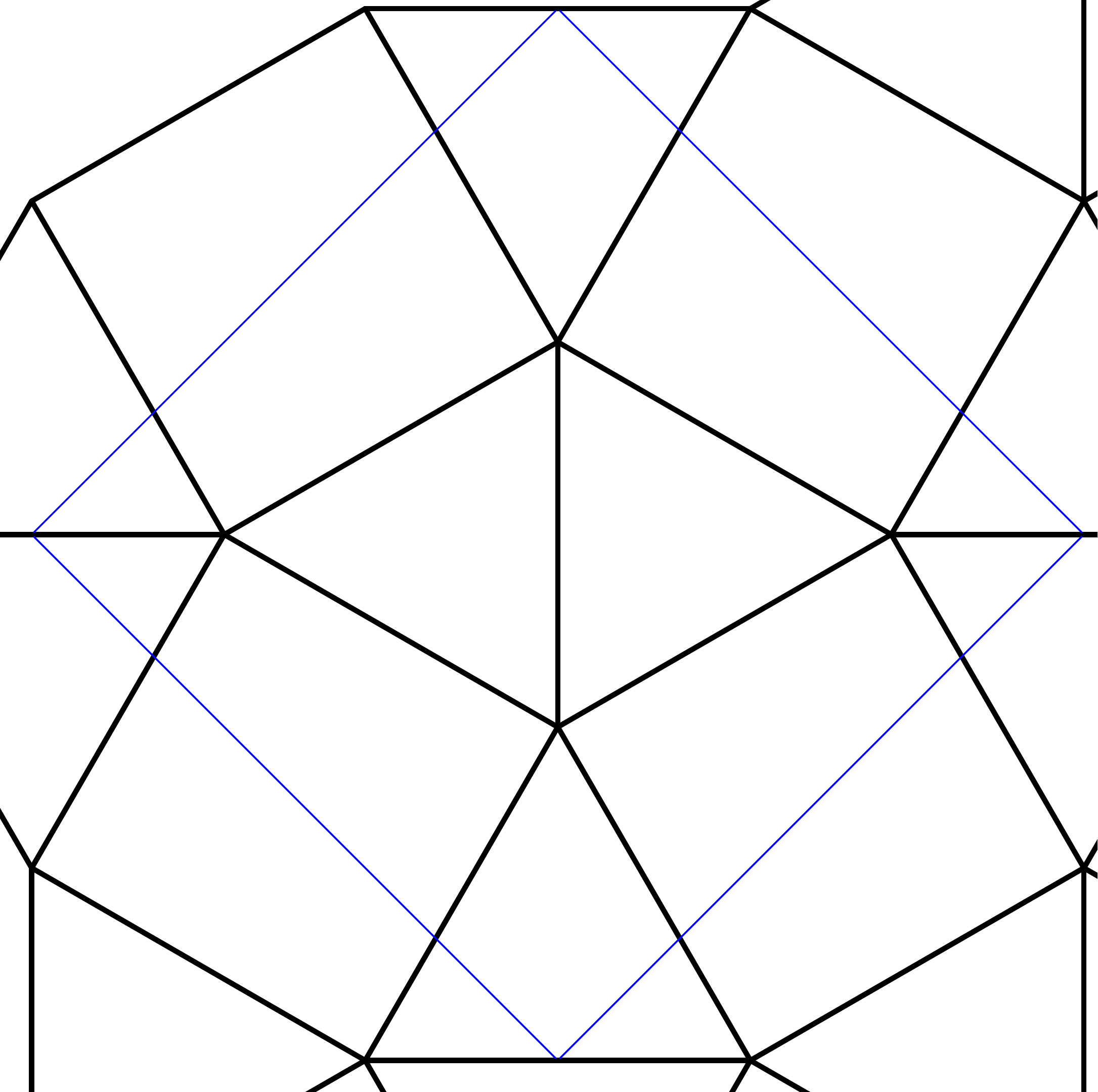}  \hspace{1cm} \includegraphics[width=40mm]{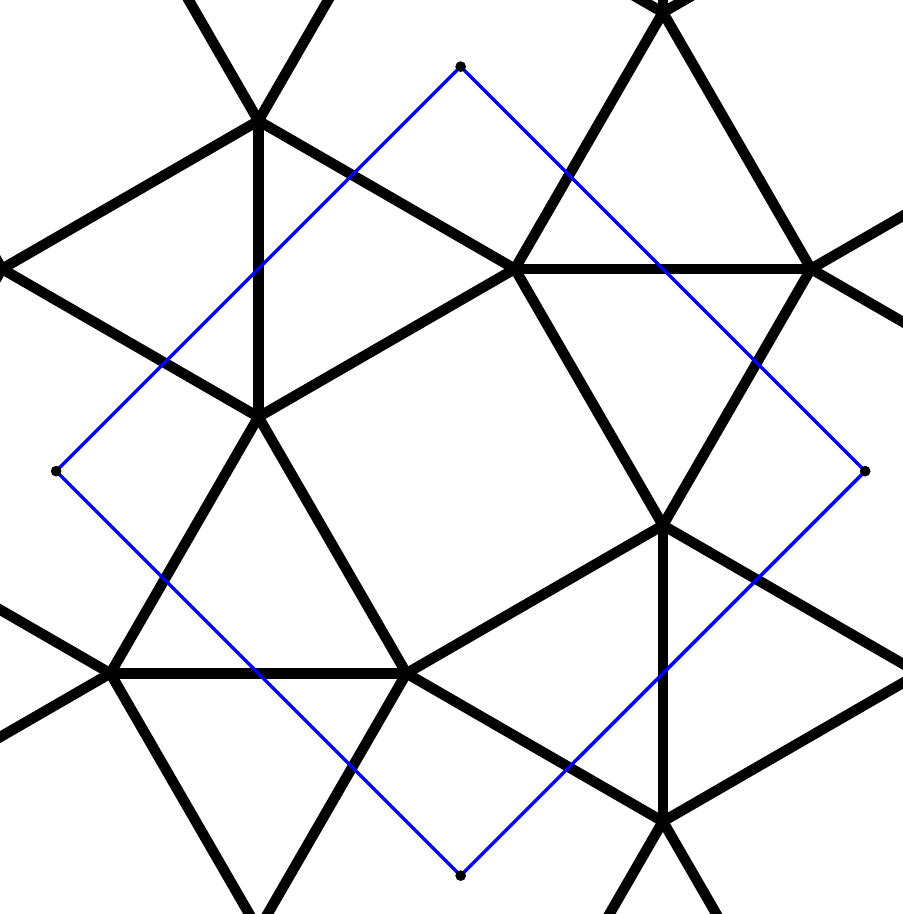}$$

\caption{Minimum number of flag orbits for $\tau$ of type  $(3^2.4.3.4)$, and a fundamental region of $\tau/ T_\tau$ drawn in two equivalent ways so that to show existence of mirror symmetries (second picture) and rotational symmetry (third picture).}
\label{f33434}
\end{figure}

Conversely, if $\tau^*/G$ is regular, then there is a rotation by ${\pi}/{2}$ around the center of a square, as well as a reflection across the edge of adjacent triangles in $\Sym(M)$. Furthermore, every translational symmetry in $\Sym(\tau/ T_\tau)$ is also in $\Sym(\tau/G)$, and thus $M$ has only five flag orbits.
\end{proof}

\begin{theorem}[Rotary to almost regular toroidal map]
\label{Thm:RotAlmostRegular}
Let $\tau$ be the Archimedean tessellation of type $(3^4.6)$. Then $\tau_{\bold{u}, \bold{v}}$ is an almost regular Archimedean map (with ten flag orbits) if and only if $\tau^*_{\bold{u}, \bold{v}}$ is a rotary map on the torus. 
\end{theorem}
\begin{proof}

Suppose that the map $\tau_{\bold{u}, \bold{v}}$ has exactly ten flag orbits. For this to be the case, there must be a rotation by ${\pi}/{3}$ around the center of a hexagon in $\Sym(\tau_{\bold{u}, \bold{v}})$. The existence of this symmetry forces $\tau^*/\left<\bold{u}, \bold{v}\right>$ to be rotary. Note that no additional reflexive symmetries are in $\Sym(\tau_{\bold{u}, \bold{v}})$.  

\begin{figure}[h]
$$\includegraphics[width=50mm]{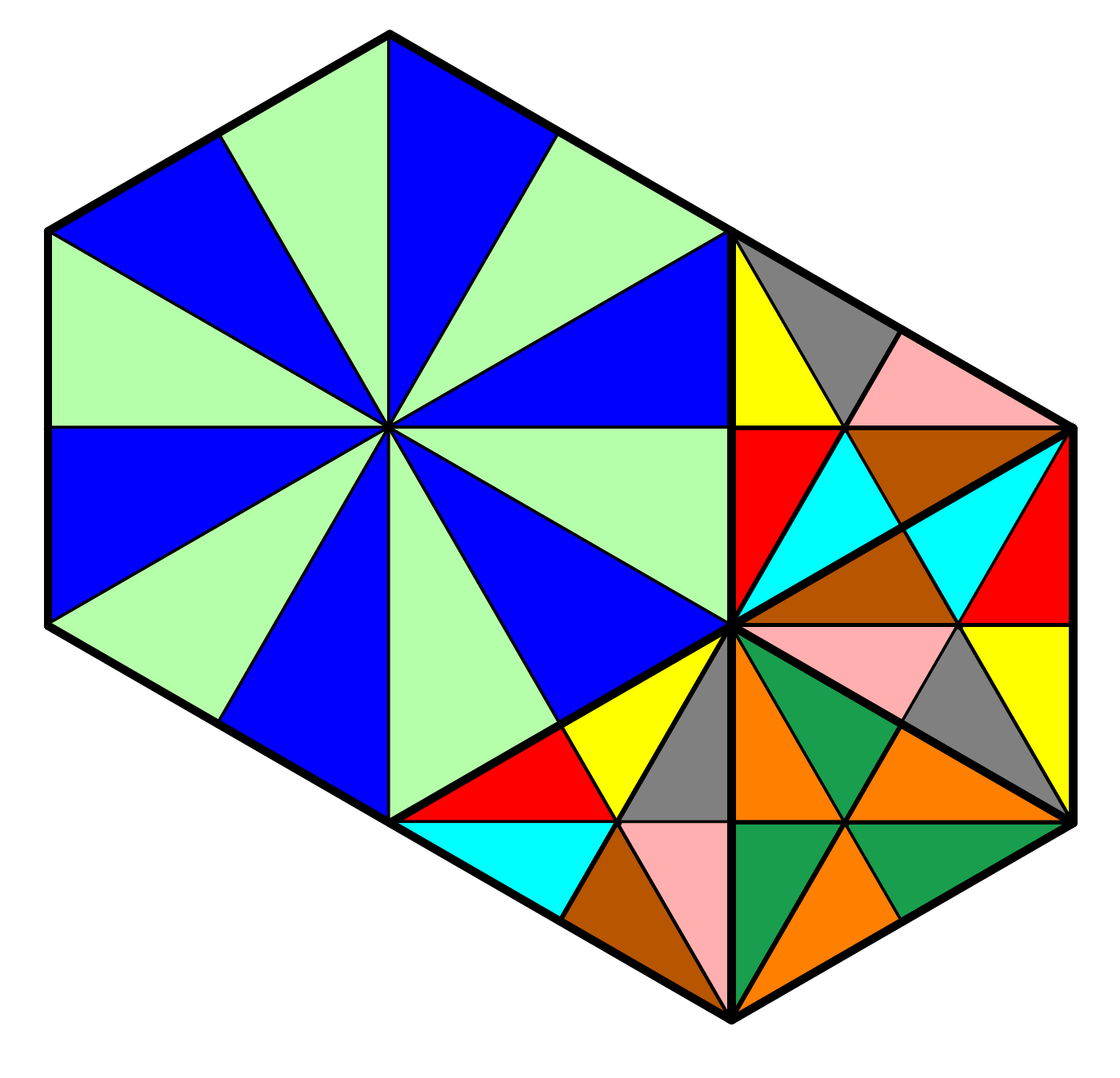} \hspace{1cm} \includegraphics[width=50mm]{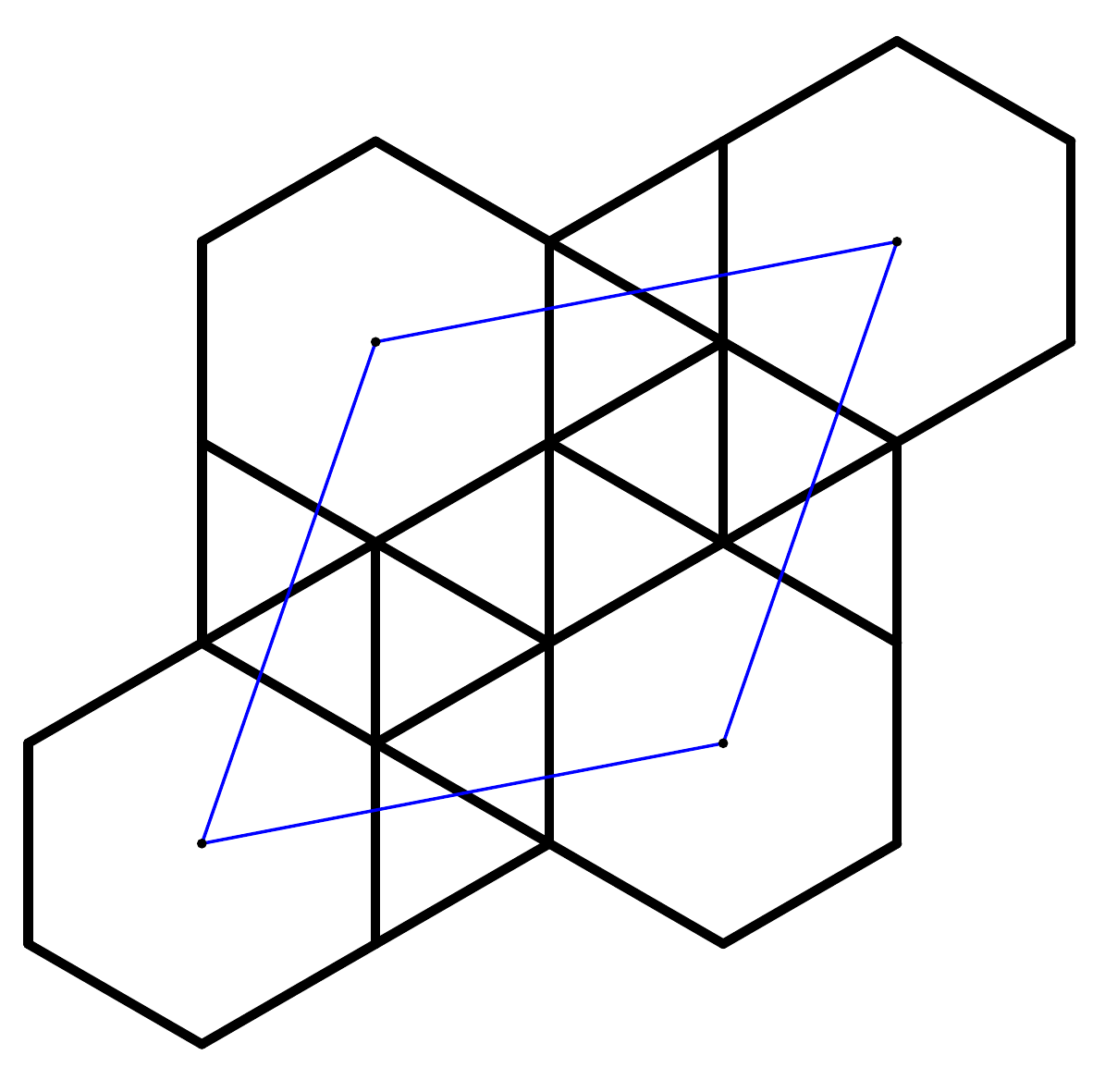}$$
\caption{ Minimum number of flag orbits for $\tau$ of type  $(3^4.6)$, and a fundamental region of $\tau/H$.}
\label{f33336}
\end{figure}

Conversely, if $\tau^*_{\bold{u}, \bold{v}}$ is rotary, then there is a rotation by ${\pi}/{3}$ around the center of a hexagon in $\Sym(\tau_{\bold{u}, \bold{v}})$. Furthermore, every translational symmetry in $\Sym(\tau/ T_\tau)$ is also in $\Sym(\tau_{\bold{u}, \bold{v}})$, and thus $\tau_{\bold{u}, \bold{v}}$ has only ten flag orbits.  Note that the translations in $\Sym(\tau/ T_\tau)$ act on the flags in 60 flag orbits, and then the listed symmetry forces there to only be ten orbits.
\end{proof}

Theorems~\ref{Thm:RegAlmostRegular} and~\ref{Thm:RotAlmostRegular} provide us with a fair understanding of how almost regular maps of type $\mathcal A_{\text{reg}}$ and $(3^4.6)$ look: for each of them the associated map on the torus must be regular, respectfully rotary.

The only remaining Archimedean tessellation not covered by the previous two results is $(3^3.4^2)$. Since the translation subgroup of the symmetry group of $(3^3.4^2)$ does not coincide with the symmetry group of one of the regular planar tessellations (as it is for all tessellations in $\mathcal A_{\text{reg}}$), or does not contain the rotation subgroup of a regular planar tessellation (as it is for the tessellation $(3^4.6)$), we have to deal with $(3^3.4^2)$ separately and with different techniques. 

In order to state a complete characterization of almost regular maps of type $(3^3.4^2)$, we introduce the following notation: write $(\bold{e_1}, \bold{e_2})$ for the positively-oriented basis of the plane $\mathbb E^2$ represented by the shortest non-parallel translations that are in the symmetry group of $(3^3.4^2)$. Recall also that $T_\tau$ stands for the maximal translation subgroup of the symmetry group $\Sym(\tau)$ of a given Archimedean tessellation $\tau$. 

\begin{theorem}[Almost regular maps of type $(3^3.4^2)$]
\label{Thm:AlmostRegular33344}
Let $\tau = (3^3.4^2)$. Then $\calM = \tau / G$, with $G < T_\tau$, is an almost regular Archimedean map (with five flag orbits) if and only if the translation subgroup $G$ is of the form 
\begin{equation}
\label{Eq:TwoTypes}
\big<c \bold{e_1}, -d \bold{e_1} + 2d \bold{e_2}\big>, \text{ or }\big<c \bold{e_1}, -d \bold{e_1} + (c+2d) \bold{e_2}\big>
\end{equation}
for non-zero integers $c$ and $d$.
\end{theorem}

\begin{remark}
Observe that the statement above, in fact, does not depend on the choice of the basis for $\mathbb E^2$. The explicit coordinate form (\ref{Eq:TwoTypes}) was used only to simplify later search for minimal almost regular covers in the proof of Theorem~\ref{Thm:33344}.
\end{remark}

\begin{remark}
The groups listed in (\ref{Eq:TwoTypes}) are never isomorphic. 
\end{remark}

\begin{proof}
The beginning of proof is similar to the proofs given for the tessellations in $\mathcal A \setminus \{(3^3.4^2)\}$. Assume that a map $\calM$ has exactly five flag orbits. For this to be the case, there must be the following symmetries in $\Sym(\calM)$:
\begin{itemize}
\item 
reflections across a horizontal line through the center of a square (see Figure~\ref{sh}); 
\item 
reflections across a vertical line through the center of a square (see Figure~\ref{sv}). 
\end{itemize}
We will write $\bold{h}$, respectively $\bold{v}$, for a reflection across a horizontal, resp.\,vertical, line through the center of a square, where a horizontal line is in the direction of $ \bold{e_1}$.

\begin{figure}[ht]
\begin{multicols}{3}
\hfill
\includegraphics[width=40mm]{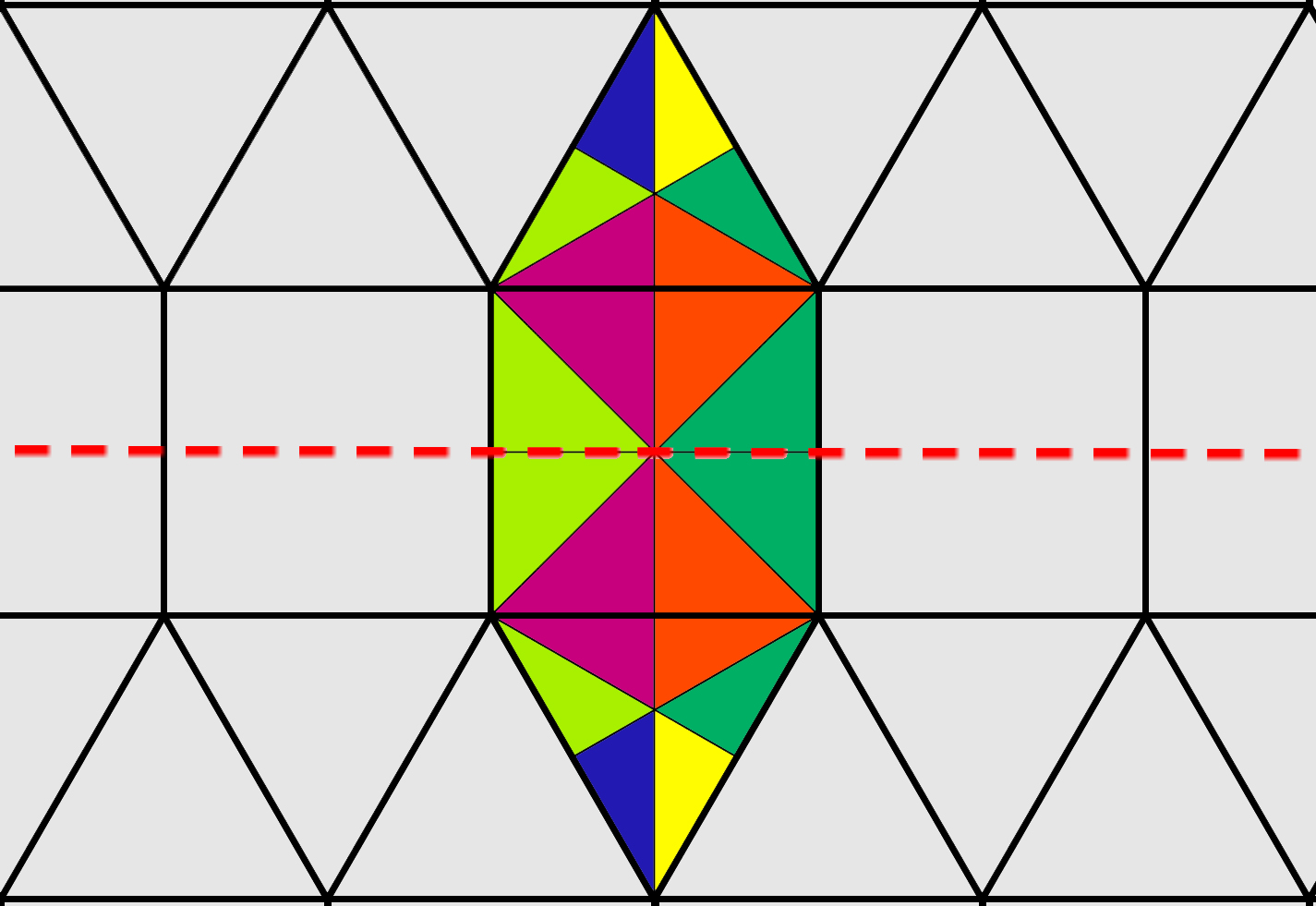}
\hfill
\caption{Reflection across the horizontal line; 10 flag orbits in the fundamental region of $(3^3.4^2) / T_{(3^3.4^2)}$.}
\label{sh}
\hfill
\includegraphics[width=40mm]{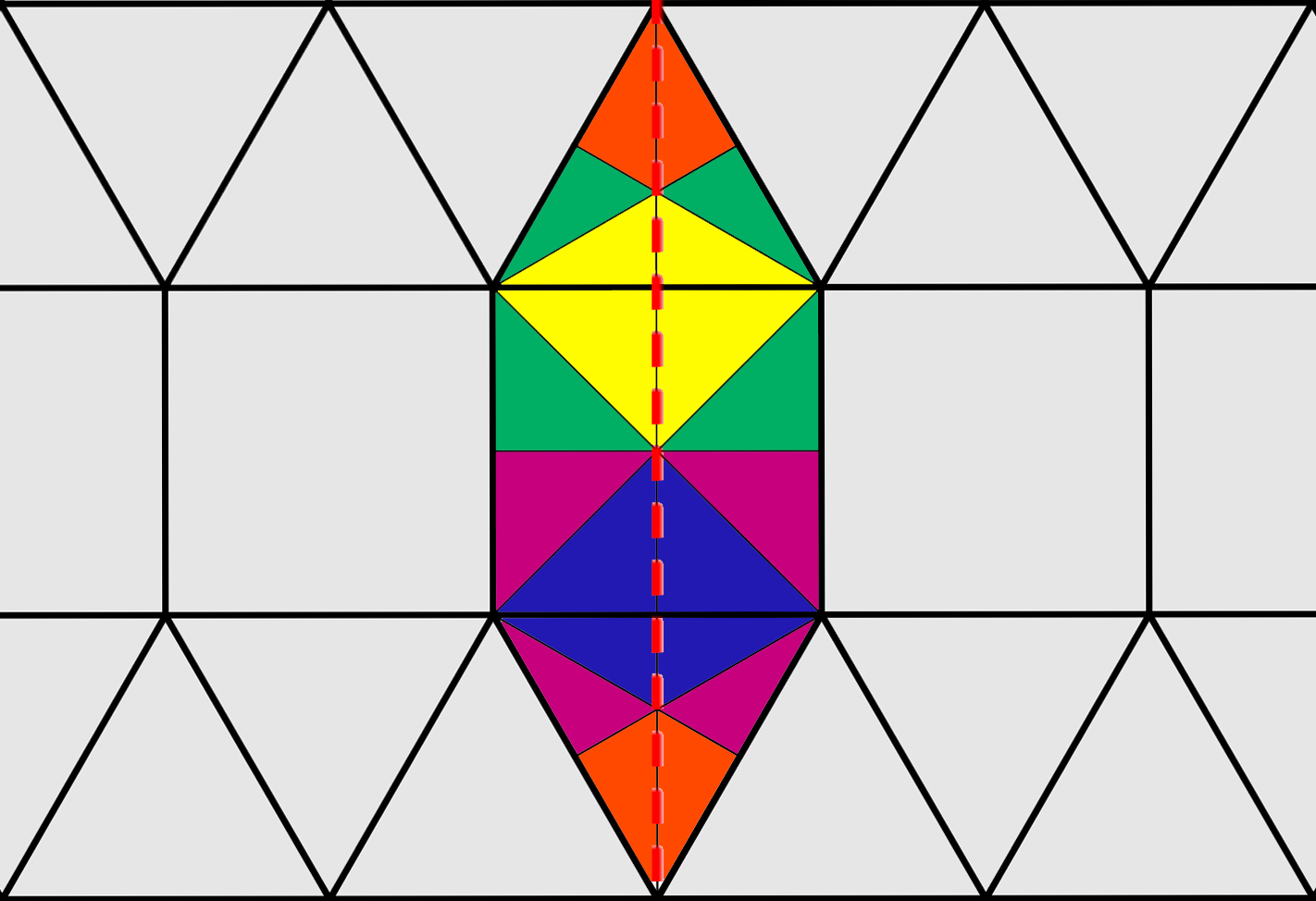}
\hfill
\caption{Reflection across the vertical line; 10 flag orbits in the fundamental region.}
\label{sv}
\hfill
\includegraphics[width=40mm]{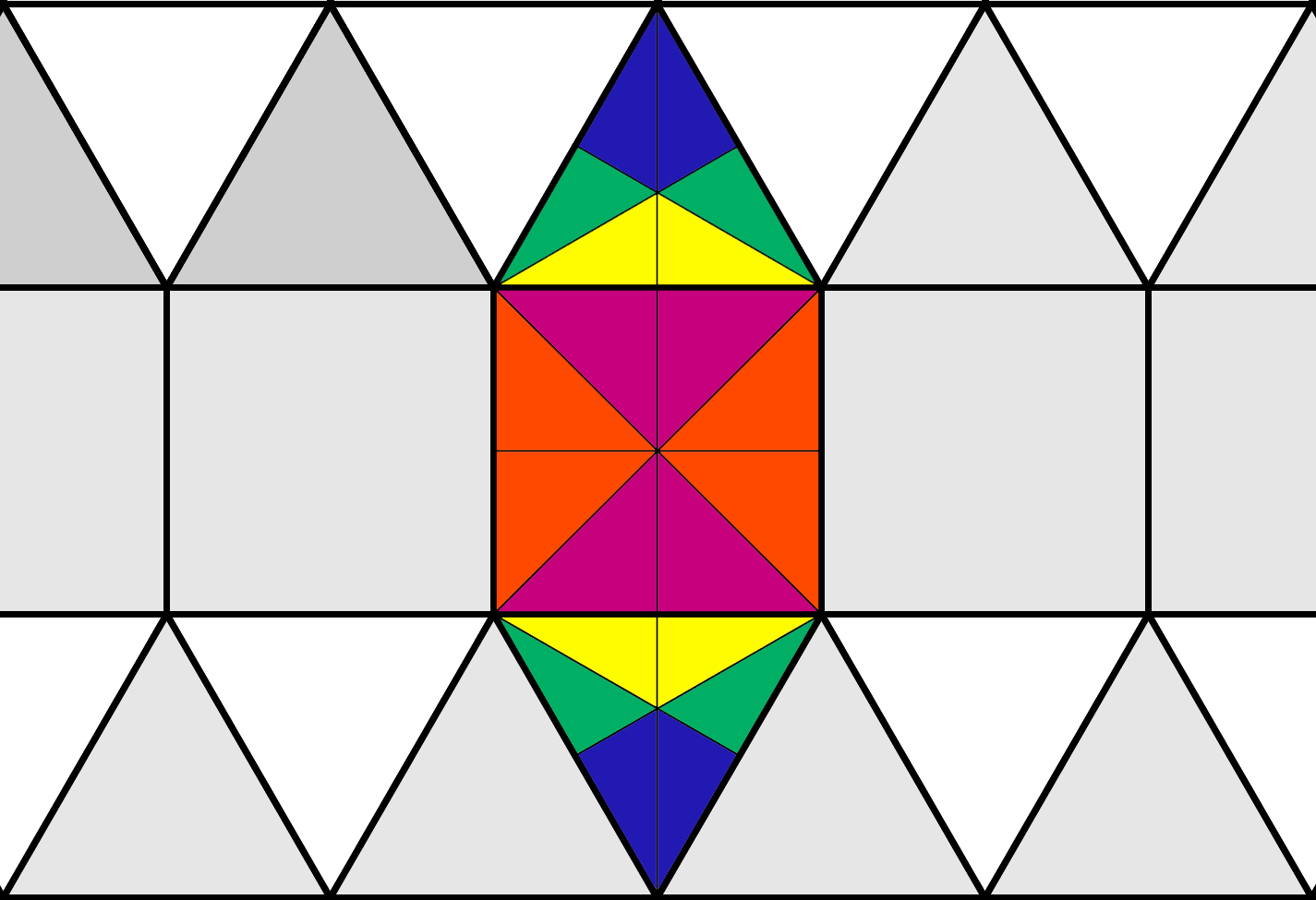}
\hfill
\caption{The minimal number of flag orbits in the fundamental region.}
\label{sv}
\end{multicols}
\end{figure}

Let us find all possible subgroups $G < T_\tau$ such that the listed symmetries preserve $G$ by conjugation. Suppose $G$ is generated by a pair of non-parallel vectors $\bold{a}, \bold{b} \in T_\tau$, and assume that $\bold{h} \circ \bold{u} \circ \bold{h}^{-1} \in G$ and $\bold{v} \circ \bold{u} \circ \bold{v}^{-1} \in G$ for every $\bold{u} \in G$.

Because the basis $(\bold{e_1}, \bold{e_2})$ of $\mathbb E^2$ was chosen in such a way that both $\bold{e_1}$ and $\bold{e_2}$ are the symmetries of $\tau$ which generate the group $T_\tau$ (i.e. $\left<\bold{e_1}, \bold{e_2}\right> = T_\tau$), there exist two pairs of integers $a_1, a_2$ and $b_1, b_2$ such that 
$$
\bold{a} = a_1 \bold{e_1} + a_2 \bold{e_2} = (a_1, a_2), \quad \bold{b} = b_1 \bold{e_1} + b_2 \bold{e_2} = (b_1, b_2).
$$
Note that if $2a_1 + a_2=2b_1 + b_2=0$, then $\bold{a}$ and $\bold{b}$ are parallel, which is impossible. Hence, without loss of generality we can assume $2a_1 + a_2 \not= 0$; this technical assumption will be used later in the proof. 

In order to understand the structure of the group $G$, observe that
\begin{equation}
\label{Eq:ConjugationOnBasis}
\begin{aligned}
&R_v(\bold{e_1}) := \bold{v} \circ \bold{e_1} \circ \bold{v}^{-1} = -\bold{e_1},\quad &R_v(\bold{e_2}) = \bold{e_2} - \bold{e_1},\\
&R_h(\bold{e_1}) := \bold{h} \circ \bold{e_1} \circ \bold{h}^{-1} = \bold{e_1},\quad &R_h(\bold{e_2}) = \bold{e_1} - \bold{e_2}.
\end{aligned}
\end{equation}
For an element $\bold{u} \in T_\tau$, the action of $R_v$ and $R_h$ on $\bold{u}$ is defined using (\ref{Eq:ConjugationOnBasis}) by linearity.

Since the reflection across a vertical line through the center of a square preserves $G$ by conjugation, the group $G$ must contain the vector $R_v (\bold{a}) = (-a_1 - a_2, a_2)$, as it is easy to compute from (\ref{Eq:ConjugationOnBasis}). Hence $G$ contains the vector $\bold{a} - R_v (\bold{a}) = (2a_1+a_2,0)$, which is not zero since $2a_1 + a_2 \neq 0$ by our assumption. Therefore, $G$ contains a proper non-trivial subgroup $G'$ of vectors with vanishing second coordinate. Pick $\bold{c} = (c_1, 0)$ with $c_1 > 0$ to be a generator of $G'$. Observe that, in fact, $\bold{c}$ is the shortest vector among all vectors in $G$ with positive first coordinate and vanishing second coordinate. 

Let $\bold{d} \in G$ be a vector such that  $G=\left<\bold{c}, \bold{d}\right>$. Moreover, since $\left<\bold{c}, \bold{d}\right> =\left<\bold{c}, \bold{d} +k \bold{c}\right>$ for every $k \in  \mathbb Z$, by picking an appropriate $k$ we may assume that $\bold{d}$ is chosen in such a way that in coordinates $\bold{d} = (d_1, d_2)$ we have 
\begin{equation}
\label{Eq:CondOnd_1}
d_1 \in \left[-\frac{d_2}{2}, c_1 - \frac{d_2}{2}\right).
\end{equation} 

Now we use the second symmetry from our list: since the conjugation by a reflection across a horizontal line through the center of a square preserves $G$, the vectors $(2d_1 + d_2, 0)$ and $(2d_1 + d_2 - c_1, 0)$ must be in $G'$. Indeed, (\ref{Eq:ConjugationOnBasis}) and linearity of $R_h$ gives us $R_h\left(\bold{d}\right) = (d_1+d_2, -d_2)$, and this vector must be in $G$. Hence, $R_h\left(\bold{d}\right) + \bold{d} = (2d_1+d_2, 0) \in G' < G$. Similarly, $R_h\left(\bold{d}\right) + \bold{d} - \bold{c} = (2d_1 +d_2 - c_1, 0) \in G'< G$. 

Recall that $\bold{c} = (c_1,0)$ was the shortest vector in $G'$ among all vectors with positive first coordinate. Therefore, as $(2d_1 +d_2, 0) \in G'$ we must have either $2d_1 + d_2= 0$, or $2 d_1 + d_2 \geqslant c_1$ (note that $2d_1 + d_2$ is necessarily non-negative by assumption (\ref{Eq:CondOnd_1})). In the latter case it then follows from (\ref{Eq:CondOnd_1}) that $0 \leqslant 2d_1 + d_2 - c_1 < c_1$, which is possible, again by minimality of $c_1$, only if $2d_1 + d_2 - c_1 = 0$. 

Therefore, either $d_2 = -2d_1$, or $d_2 = -2 d_1 + c_1$, and hence setting $c:=c_1$, $d:=-d$ we obtain that the group $G$ must be of one of the types in (\ref{Eq:TwoTypes}). One implication in Theorem~\ref{Thm:AlmostRegular33344} is proven.

Conversely, it is straightforward to see that a group of one of the types in (\ref{Eq:TwoTypes}) is preserved by conjugation with both $\bold{h}$ and $\bold{v}$, and hence the symmetry group $\Sym(\tau/G)$ contains both $\bold{h}$ and $\bold{v}$, which implies that $\tau / G$ has only five flag orbits, and thus is an almost regular Archimedean map.   
\end{proof}

\section{Minimal covers of Archimedean toroidal maps}
\label{Sec:MinCov}

In this section we will prove the Main Theorem. This will be done by combination of three statements -- Theorem~\ref{Thm:Areg}, \ref{Thm:Arot} and \ref{Thm:33344}. Prior the proofs, in the following subsection we recall some of the facts about the Gaussian and the Eisenstein integers -- a number-theoretic tool that provides a natural `language' for our main results.

\subsection{Gaussian and Eisenstein integers.}
\label{Subsec:Gauss}

The Gaussian and Eisenstein integers provide an essential ingredient for understanding the Archimedean toroidal maps.  The plots of these domains in the complex plane are the vertex sets of regular tessellations: a tessellation of type $\{4,4\}$ for the Gaussian integers and $\{3,6\}$ for the Eisenstein integers. We will follow \cite{DMEquivelar} and \cite{ Ireland} in our notation.

The \textit{Gaussian integers} $\mathbb Z[i] $ are defined as the set $\left\{a+bi \colon a, b \in \mathbb Z\right\} \subset \mathbb C$, where $i$ is the imaginary unit, with the standard addition and multiplication of complex numbers. Similarly, the \textit{Eisenstein integers} $\mathbb Z[\omega]$ are defined as $\left\{a+b\omega \colon a, b \in \mathbb Z\right\}\subset \mathbb C$, where $\omega := {(1 + i\sqrt{3})}/{2}$. We adopt a unifying notation $\mathbb Z[\sigma]$ with $\sigma \in\{i, \omega\}$ to denote either of these two sets.  

As in~\cite{DMEquivelar}, in this paper we use the following notation:  given $\alpha = a + b\sigma \in \mathbb Z[\sigma]$, we call its \textit{conjugate} the number $\overline\alpha := a + b \overline \sigma$, where $\overline \sigma$ is the conjugate complex number to $\sigma \in \mathbb{C}$.  Also we call $\R \alpha := a$ and $\I \alpha := b$ the \textit{real} and \textit{imaginary} parts, respectively. Note here that if $\sigma = i$, then this is the traditional notion of  `real part' and `imaginary part' of a complex number.  However, if $\sigma = \omega$, then the `traditional' real and imaginary parts of $a+b\omega$ are $a+b/2$ and $b\sqrt{3}/2$, respectively. 

For every $\alpha = a+b\sigma \in \mathbb Z[\sigma]$, we assign the \textit{norm} $N(\alpha) := \alpha \overline\alpha$. An integer $\alpha \in \mathbb Z[\sigma] \backslash \{0\}$ \textit{divides} $\beta \in \mathbb Z[\sigma]$ if and only if there is $\gamma \in \mathbb Z[\sigma]$ such that $\beta = \alpha\gamma$. Recall that, in the ring of Gaussian integers, the \emph{units} are only $\pm 1, \pm i$, while in the ring of Eisenstein integers the units are only $\pm 1, \pm \omega, \pm \overline \omega$. Two integers $\alpha, \beta \in \mathbb Z[\sigma]$ are called \textit{associated} if $\alpha = \beta\varepsilon$ for some unit $\varepsilon$. 

Let us recall the concept of a greatest common divisor for rings of integers. Am integers $\gamma \in \mathbb Z[\sigma]$ is a \textit{greatest common divisor (GCD)} of $\alpha, \beta \in \mathbb Z[\sigma]$ with $N(\alpha) + N(\beta) \neq 0$, if $\gamma$ divides both $\alpha$ and $\beta$ and for every $\gamma'$ with the same property it follows that $\gamma'$ divides $\gamma$. It is well-known that both $\mathbb Z[i]$ and $\mathbb Z[\omega]$ are Unique Factorization Domains (see \cite{Ireland}), that is the rings with unique (up to associates) factorization into primes. Hence, a greatest common divisor is well-defined, again up to associates. Because of that, we write $\gamma = \GCD(\alpha,\beta)$ implying that $\gamma$ is defined up to multiplication by an associate. We also agree that if there is a rational integer $n$ among associates to $\GCD(\alpha,\beta)$, then we specifically take $\GCD(\alpha,\beta) := |n|$. For example, $\GCD(3,6i) \in \{3,-3,3i,-3i\}$, and thus by our convention $\GCD(3,6i) = 3$.

The power of Gaussian and Eisenstein integers is coming from the natural identification of these sets with the vertex set of a regular tessellation $\{4,4\}$ or $\{3,6\}$. In particular, we can identify the basis $(\bold{e_1}, \bold{e_2})$ (see Section~\ref{Sec:Prelim}) with the ordered pair $(1, \sigma)$ from $\mathbb Z[\sigma]$. This identification leads to the group homomorphism of $T_{\{4,4\}}$ resp.\, $T_{\{3,6\}}$ with $\mathbb Z[i]$ resp.\, $\mathbb Z[\omega]$ (where the latter groups are regarded as Abelian groups with respect to addition). From this point of view, we will identify every two vectors $\bold a = (a_1, a_2)$, $\bold b = (b_1, b_2)$ with the complex numbers $\alpha =a_1 +  a_2\sigma$ and $\beta = b_1 +  b_2\sigma$.  Finally, write $\tau_{\alpha,\beta} := \tau_{\bold a, \bold b}$, and, for further brevity, $\tau_\eta := \tau_{\eta,\sigma\eta}$. 

Note here that the pair of vectors $\eta$, $\sigma \eta$ span a square if $\sigma = i$, or a rhombus with angle $\pi/3$ if $\sigma = \omega$. Therefore, if $\eta$ and $\eta'$ are equal to to associates, then $\tau_\eta = \tau_{\eta'}$.

\subsection{Proof of the Main Theorem.}

In Theorems \ref{Thm:Areg}~-- \ref{Thm:33344} we will prove that each Archimedean map on the torus has a unique minimal almost regular cover on the torus, which we will construct explicitly. To accomplish these proofs, we will use some known results for equivelar toroidal maps (see~\cite{DMEquivelar}).

\begin{proposition}[Covering correspondence of maps and their associates] 
\label{Prop:CoveringCorresp}
Let $\tau$ be an Archimedean tessellation of the plane, not of type $(3^4.6)$, and $G$ and $H$ two subgroups of $T_\tau$.  Then $\tau/H$ covers $\tau/G$ if and only if $\tau^*/H$ covers $\tau^*/G$.

\end{proposition}

\begin{proof}
Let $\calN= \tau / H $ and $\calM=  \tau / G$ be maps on the torus, where $H$ is a subgroup of $G$, and $G$ is a subgroup of $T_\tau$.  By definition, thus $\mathcal N \searrow \mathcal M$.  By construction the groups $T_\tau$ and $T_{\tau^*}$ are isomorphic, and so the same subgroup structure holds in $T_{\tau^*}$, which means that $\tau^*/H$ covers $\tau^*/G$. Converse is similar.
\end{proof}


Observe that for every $\tau \in \mathcal A \setminus \{(3^3.4^2)\}$ we have a well-defined basis $\left(\bold{e_1}, \bold{e_2}\right)$ such that $\left<\bold{e_1}, \bold{e_2}\right> = T_\tau = T_{\tau^*}$, where $\tau^* = \{p, q\}$ is the regular tessellation associated to $\tau$. Therefore, every $\bold{u} \in T_\tau$ with the coordinates $(u_1, u_2)$ in the basis $(\bold{e_1}, \bold{e_2})$ can be identified by the homomorphism discussed in Subsection~\ref{Subsec:Gauss} with the integer $u_1 + u_2 \sigma \in \mathbb Z[\sigma]$ (here $\sigma$ depends on $\tau^*$). We will use this equivalent language instead of the vector language in order to state the following two theorems which are the first two main results of the paper.

\begin{theorem}[Minimal almost regular covers for toroidal maps of type $\mathcal A_{\text{reg}}$]
\label{Thm:Areg}
Let $ \tau_{\alpha, \beta}$ be an Archimedean map given as a quotient of an Archimedean tessellation $\tau \in \mathcal A_{\text{reg}}$ by a translation subgroup $\left<\alpha, \beta\right> < T_\tau$ generated by two integers $\alpha, \beta \in \mathbb{Z}[\sigma] \setminus \{0\}$ with $\alpha/ \beta \not\in \mathbb Z$. Then the map $\tau_\eta$ with 
\begin{equation*}
\eta=\left\{
\begin{aligned}
&\frac{\I \left(\overline{\alpha}\beta\right) }{N(1+\sigma)\,c} \, (1+\sigma), \text{ if }N(1 + \sigma) \text{ divides } \frac{\R \alpha}{c} - \frac{\I \alpha}{c} \text{ and } \frac{\R \beta}{c} - \frac{\I \beta}{c},\\
&\frac{\I (\overline{\alpha}\beta)}{c} , \text{ otherwise},
\end{aligned}
\right.
\end{equation*}
where $c = \GCD\left(\R \alpha, \I \alpha, \R \beta, \I \beta\right)$ is a unique minimal almost regular cover of $\tau_{\alpha, \beta}$. Moreover, the number $K_{\min}$ of fundamentals regions of $\tau_{\alpha, \beta}$ glued together in order to the fundamental region of $\tau_\eta$ is equal to
\begin{equation*}
K_{\min}=\left\{
\begin{aligned}
&\frac{\left|\I \left(\overline{\alpha}\beta\right)\right|}{N(1+\sigma)\, c^2}, \text{ if }N(1 + \sigma) \text{ divides } \frac{\R \alpha}{c} - \frac{\I \alpha}{c} \text{ and } \frac{\R \beta}{c} - \frac{\I \beta}{c},\\
&\frac{\left|\I (\overline{\alpha}\beta)\right|}{c^2} , \text{ otherwise}.
\end{aligned}
\right.
\end{equation*}
\end{theorem}
\begin{proof}
This theorem is a direct consequence of \cite[Theorem 3.6]{DMEquivelar}. Indeed, let $\tau^*$ be the regular tessellation associated to the Archimedean tessellation $\tau$. By Theorem~\ref{Thm:RegAlmostRegular}, there is one-to-one correspondence between the translation subgroups of $T_{\tau^*}$ and $T_\tau$ that generate regular resp.\,\,almost regular maps on the torus. By Proposition~\ref{Prop:CoveringCorresp}, any such correspondence preserves the covering order and, it particular, sends minimal elements to minimal elements. Therefore, $\tau^*/H$ is the minimal regular cover of $\tau^*/G$, where $G := \left<\alpha, \beta\right>$, if and only if $\tau / H$ is the minimal almost regular cover of $\tau / G$. By \cite[Theorem 3.6]{DMEquivelar}, every map $\tau^*/G$ has a unique minimal regular cover $\tau^*/H$, where $H < T_{\tau^*}$ can be given explicitly in terms of number-theoretical properties of $\alpha, \beta \in \mathbb Z[\sigma]$. Hence, the same holds for $\tau/G$, which yields existence and uniqueness of a minimal almost regular cover. The explicit form of $H$ from \cite[Theorem 3.6]{DMEquivelar} translates verbatim into the explicit form given in Theorem~\ref{Thm:Areg}; this finishes the proof.  
\end{proof}

\begin{theorem}[Minimal almost regular covers for toroidal maps of type $(3^4.6)$]
\label{Thm:Arot}
Let $ \tau_{\alpha, \beta}$ be an Archimedean map given as a quotient of an Archimedean tessellation $\tau = (3^4.6)$ by a translation subgroup $\left<\alpha, \beta\right> < T_\tau$ generated by two integers $\alpha, \beta \in \mathbb{Z}[\omega] \setminus \{0\}$ with $\alpha/ \beta \not\in \mathbb Z$. Then the map $\tau_\eta$ with 
$$
\eta = \frac {\I (\overline{\alpha}\beta)}{N(\gamma)}\gamma
$$
where $\gamma = \GCD(\alpha, \beta)$ is a unique minimal almost regular cover of $\tau_{\alpha, \beta}$. Moreover, the number $K_{\min}$ of fundamentals regions of $ \tau_{\alpha, \beta}$ glued together in order to obtain the fundamental region of $\tau_\eta$ is equal to
$$ 
K_{\min} = \frac{\left|\I (\overline{\alpha} \beta) \right|}{N(\gamma)}.
$$
\end{theorem}

\begin{proof}
The proof is similar to the proof of Theorem~\ref{Thm:Areg}, where instead of \cite[Theorem 3.6]{DMEquivelar} we use \cite[Theorem 3.5]{DMEquivelar}.
\end{proof}

Recall that we associate to the tessellation of type $(3^3.4^2)$ the basis $(\bold{e_1}, \bold{e_2})$ comprised of a pair of shortest non-parallel translations that are in the symmetry group of $(3^3.4^2)$. Everywhere below we assume that the coordinate representation of a translation from  $T_{(3^3.4^2)}$ is given with respect to the basis $(\bold{e_1}, \bold{e_2})$.     

\begin{theorem}[Minimal almost regular covers for toroidal maps of type $(3^3.4^2)$] 
\label{Thm:33344}
Let $\tau_{\bold{a}, \bold{b}}$ be an Archimedean toroidal map given as a quotient of the tessellation $\tau=(3^3.4^2)$ by a translation subgroup $\left<\bold{a}, \bold{b}\right> < T_\tau$ generated by two vectors $\bold{a} = (a_1, a_2)$ and $\bold{b} = (b_1,b_2)$ with $\Delta:= a_1 b_2 - a_2 b_1 \neq 0$. Then for $\tau_{\bold{a}, \bold{b}}$ there exists and is unique a minimal almost regular cover $\tau_{\bold{u}, \bold{v}}$ generated by the subgroup
\begin{equation*}
\left<\bold{u},\bold{v}\right>=\left\{
\begin{aligned}
&\big<(c, 0), (-d_1, c + 2d_1)\big>, \text{ provided } \frac{a_2}{g_1}-\frac{2a_1+a_2}{g_2} \text{ and } \frac{b_2}{g_1}-\frac{2b_1+b_2}{g_2} \text{ are even}, \\
&\big<(c, 0), (-d_2, 2d_2)\big>, \text{ otherwise}, \\
\end{aligned}
\right.
\end{equation*}
where 
\begin{equation*}
\begin{aligned}
&g_1 = \GCD(a_2, b_2), \quad g_2 = \GCD(2a_1+a_2, 2b_1+b_2),\\
&c=\frac{\Delta}{g_1}, \quad d_1=-\frac{\Delta}{2}\left(\frac{1}{g_1} + \frac{1}{g_2}\right), \quad d_2 = -\frac{\Delta}{g_2}.
\end{aligned}
\end{equation*}
Moreover, the number $K_{\min}$ of fundamentals domains of $\tau_{\bold{a}, \bold{b}}$ glued together in order to obtain the fundamental region of $\tau_{\bold{u}, \bold{v}}$ is equal to
\begin{equation*}
K_{\min} = \left\{
\begin{aligned}
&\left|\frac{\Delta}{g_1 g_2}\right|, \text{ if } \frac{a_2}{g_1}-\frac{2a_1+a_2}{g_2} \text{ and } \frac{b_2}{g_1}-\frac{2b_1+b_2}{g_2} \text{ are even}, \\
&2\left|\frac{\Delta}{g_1 g_2}\right|, \text{ otherwise}. \\
\end{aligned}
\right.
\end{equation*}
\end{theorem}

\begin{proof}

Our strategy in proving Theorem~\ref{Thm:33344} will be the following: we explicitly describe \emph{all} almost regular covers for $\tau_{\bold{a}, \bold{b}}$ and then determine the one which is the smallest (under the covering relation). 

Suppose that an Archimedean map $\tau_{\bold{u}, \bold{v}}$ is a cover of $\tau_{\bold{a}, \bold{b}}$. This is equivalent of saying that the group $\left<\bold{u}, \bold{v}\right>$ is a proper subgroup of $\left<\bold{a}, \bold{b}\right>$, which in algebraic terms is equivalent to existence of a linear integer relation between the generators of both groups:
\begin{equation}
\label{Eq:CoverRel}
\left\{
\begin{aligned}
&n_1 \bold{a} + m_1 \bold{b} = \bold{u},\\ 
&n_2 \bold{a} + m_2 \bold{b} = \bold{v},
\end{aligned}
\right.
\end{equation}
for some integers $n_1,n_2,m_1,m_2$ with $n_1m_2 \neq n_2m_1$. (The last condition guarantees that $\bold{u}$ and $\bold{v}$ are, in fact, non-parallel.) If, on top, $\tau_{\bold{u},\bold{v}}$ is almost regular, then the generators $\bold{u}, \bold{v}$ might be chosen to be of one of the types in (\ref{Eq:TwoTypes}) (see Theorem~\ref{Thm:AlmostRegular33344}). We now consider these two cases one by one.

\textsc{Case 1:} suppose $\bold{u} = c \bold{e_1}$ and $\bold{v} = -d\bold{e_1} + 2d \bold{e_2}$ for some non-zero integers $c$ and $d$. Then, in order to find the generators of such type, we must solve the following system of vector Diophantine equations
\begin{equation}
\label{Eq:Case1General}
\left\{
\begin{aligned}
&n_1 \bold{a} + m_1 \bold{b} = c \bold{e_1},\\ 
&n_2 \bold{a} + m_2 \bold{b} = -d\bold{e_1} + 2d \bold{e_2},
\end{aligned}
\right.
\end{equation}
for the variables $n_1, n_2, m_1, m_2$ treating $c$ and $d$ as parameters.

The first equation in (\ref{Eq:Case1General}) in coordinates is equivalent to the system of linear Diophantine equations
\begin{equation}
\label{Eq:Case1FirstEq}
\left\{
\begin{aligned}
&n_1 a_1+m_1 b_1=c, \\
&n_1 a_2+m_1b_2=0.
\end{aligned}
\right.
\end{equation}
By the standard methods the full family of solutions for (\ref{Eq:Case1FirstEq}) is
\begin{equation}
\label{Eq:Case1SolFirstEq}
n_1 = \frac{b_2}{g_1} k, \quad m_1 = - \frac{a_2}{g_1} k, \quad c = \frac{\Delta}{g_1} k, \quad k \in \mathbb Z^*,
\end{equation}
where we recall that $\Delta = a_1 b_2 - a_2 b_1$ and $g_1 = \GCD(a_2, b_2)$ (here $\mathbb Z^*$ stands for the set of all non-zero integers).

Similarly, the second equation in (\ref{Eq:Case1General}) in coordinates reads
\begin{equation}
\label{Eq:Case1SecondEq}
\left\{
\begin{aligned}
&n_2 a_1+m_2 b_1=-d, \\
&n_2 a_2+m_2b_2=2d;
\end{aligned}
\right.
\end{equation}
Multiplying the first equation by $2$ and adding the second we get 
$$
n_2(2a_1+a_2)+m_2(2b_1+b_2)=0,
$$
from which we conclude, after straightforward cancellations, that the full family of solutions for (\ref{Eq:Case1SecondEq}) is
\begin{equation}
\label{Eq:Case1SolSecondEq}
n_2= \frac{2b_1+b_2}{g_2} s, \quad m_2=-\frac{2a_1+a_2}{g_2} s,  \quad d = -\frac{\Delta}{g_2}s, \quad s \in \mathbb{Z}^*,
\end{equation}
where $g_2 = \GCD(2a_1+a_2,2b_1+b_2)$.

Concluding \textsc{Case 1} from obtained solution (\ref{Eq:Case1SolFirstEq}) and (\ref{Eq:Case1SolSecondEq}), we see that $\tau_{\bold{u}, \bold{v}}$ is an almost regular Archimedean cover of $\tau_{\bold{a}, \bold{b}}$ and is of the first type in (\ref{Eq:TwoTypes}) if and only if
$$
\left<\bold{u}, \bold{v}\right> = \left<\left(\frac{\Delta}{g_1}k, 0\right), \left(\frac{\Delta}{g_2} s, -2 \frac{\Delta}{g_2}s\right)\right> =: G_{k,s}.
$$ 
Finally, observe that for ever given pair of non-zero integers $k$ and $s$ the almost regular map $\tau / G_{k,s}$ covers $\tau / G_{1,1}$. Hence, $\tau / G_{1,1}$ is the minimal almost regular cover (for $\tau_{\bold{a}, \bold{b}}$) of the first type in (\ref{Eq:TwoTypes}). Note that the full family of toroidal maps $\tau / G_{k,s}$ does not form a totally ordered set with respect to covering; however, as we saw, the corresponding poset has a unique minimal element. We will see a similar type of covering behavior later.

Finishing \textsc{Case 1}, we compute the number of fundamental regions of $\left<\bold{a}, \bold{b}\right>$ one should glue together in order to obtain the fundamental region of $G_{1,1}$. This is done by comparing areas of those regions. In the standard basis in $\mathbb E^2$ the area of the fundamental region of $\left<\bold{a}, \bold{b}\right>$ is equal to 
$$
A_0 := \left|\bold{a} \times \bold{b}\right| = |\Delta| \cdot \left|\bold{e_1} \times \bold{e_2}\right|.
$$
Similarly, by using (\ref{Eq:Case1SolFirstEq}) and (\ref{Eq:Case1SolSecondEq}) we compute the area of the fundamental region of $G_{1,1}$:
$$
A_1 := \left|c \bold{e_1} \times (-d\bold{e_1} + 2d \bold{e_2})\right| = \frac{2 \Delta^2}{\left|g_1 g_2\right|} \cdot \left|\bold{e_1} \times \bold{e_2}\right|. 
$$
Therefore, the number we are looking for is equal to 
\begin{equation}
\label{Eq:Case1Kmin}
\frac{A_1}{A_0} = 2 \left|\frac{\Delta}{g_1 g_2}\right|.
\end{equation}

Observe that 
\begin{equation}
\label{Eq:KeyRelation}
2 (a_1 b_2 - a_2 b_1) = (2a_1 + a_2) b_1 - (2b_1 + b_2)a_2,
\end{equation}
and hence the right hand side in (\ref{Eq:Case1Kmin}) is an integer. Of course, the same conclusion likewise follows from the geometric meaning of $A_1 / A_0$.

\textsc{Case 2:} suppose $\bold{u} = c \bold{e_1}$ and $\bold{v} = -d\bold{e_1} + (c+2d) \bold{e_2}$ for some non-zero integers $c$ and $d$; this is the second type in (\ref{Eq:TwoTypes}). We proceed similarly to \textsc{Case 1}, with a bit more involved computation.

Again, in order to find all almost regular covers of the second type we have to find all solutions of the system
\begin{equation*}
\left\{
\begin{aligned}
&n_1\bold{a} + m_1\bold{b} = c \bold{e_1}, \\
&n_2\bold{a} + m_2\bold{b} = -d \bold{e_1} + (c+2d) \bold{e_2},
\end{aligned}
\right.
\end{equation*}
for integers $n_1,n_2, m_1, m_2$ treating $c$ and $d$ as parameters. Similarly to \textsc{Case 1}, the solutions to the first equation in this system has the following form:
\begin{equation}
\label{Eq:Case2SolFirstEq}
n_1 = \frac{b_2}{g_1} k, \quad m_1 = - \frac{a_2}{g_1} k, \quad c = \frac{\Delta}{g_1} k, \quad k \in \mathbb Z^*.
\end{equation}

The second equation from the system above in coordinates reads:
\begin{equation}
\label{Eq:Case2SecondEq}
\left\{
\begin{aligned}
&n_2 a_1+m_2 b_1=-d, \\
&n_2 a_2+m_2b_2=c+2d.
\end{aligned}
\right.
\end{equation}
Again, multiplying the first equation by $2$ and adding the second one we obtain, by using (\ref{Eq:Case2SolFirstEq}),
\begin{equation}
\label{Eq:AuxEq}
n_2(2a_1+a_2)+m_2(2b_1+b_2)=c=\frac{\Delta}{g_1}k.
\end{equation}

This is a linear nonhomogeneous Diophantine equation in $n_2$ and $m_2$. The standard theory of linear Diophantine equations tells us that the any solution of (\ref{Eq:AuxEq}) is the sum of a partial solution to the given equation and of the general solution of the corresponding homogeneous equation $n_2(2a_1+a_2)+m_2(2b_1+b_2) = 0$. Moreover, a necessary condition for (\ref{Eq:AuxEq}) to have a solution is that 
\begin{equation}
\label{Eq:NecCond}
g_2 \text{ divides } \frac{\Delta}{g_1} k.
\end{equation}

Using (\ref{Eq:KeyRelation}) it is straightforward to check that $\Delta / (g_1 g_2) \in \mathbb Z^*$ if and only if 
\begin{equation}
\label{Eq:NecCond2}
\frac{a_2}{g_1}-\frac{2a_1+a_2}{g_2} \text{ and } \frac{b_2}{g_1}-\frac{2b_1+b_2}{g_2} \text{ are even.}
\end{equation}

Therefore, if condition (\ref{Eq:NecCond2}) is met, then the necessary condition (\ref{Eq:NecCond}) is satisfied for every integer $k$. On the other hand, if (\ref{Eq:NecCond2}) is violated, then (\ref{Eq:NecCond}) can be only satisfied if $k$ is even (as it follows from (\ref{Eq:KeyRelation})). Let us consider these two sub-cases.

\textsc{Sub-case 1: }suppose that condition (\ref{Eq:NecCond2}) is violated; hence $k$ is necessarily even. Then it is straightforward to check that the pair of integers
$$
n_2' = \frac{b_2}{2g_1} k, \quad m_2' = - \frac{a_2}{2 g_1} k
$$  
provides a partial solution of equation (\ref{Eq:AuxEq}). Therefore, in \textsc{Sub-case 1} the full family of solutions of (\ref{Eq:AuxEq}), and thus of (\ref{Eq:Case2SecondEq}), has the form
\begin{equation}
\label{Eq:Case2Subcase1Sol}
n_2 = \frac{b_2}{2 g_1} k + \frac{2b_1 + b_2}{g_2} s, \quad m_2 = - \frac{a_2}{2 g_1} k - \frac{2a_1 + a_2}{g_2} s, \quad d = - \frac{\Delta}{2 g_1} k - \frac{\Delta}{g_2} s, \quad k \in 2 \mathbb Z, \quad s \in \mathbb Z.
\end{equation}

Concluding \textsc{Sub-case 1} of \textsc{Case 2} by combining (\ref{Eq:Case2SolFirstEq}) and (\ref{Eq:Case2Subcase1Sol}), we see that, provided condition (\ref{Eq:NecCond2}) is not met, $\tau_{\bold{u}, \bold{v}}$ is an almost regular Archimedean cover of $\tau_{\bold{a}, \bold{b}}$ and is of the second type in (\ref{Eq:TwoTypes}) if and only if
$$
\left<\bold{u}, \bold{v}\right> = \left<\left(2\frac{\Delta}{g_1}l, 0\right), \left(\frac{\Delta}{g_1} l + \frac{\Delta}{g_2} s, -2 \frac{\Delta}{g_2}s\right)\right> =: H_{l,s},
$$ 
where $l, s \in \mathbb Z^*$.

Now let us check the covering relations. Note that for a given pair of non-zero integers $l$ and $s$, the almost regular map $\tau / H_{l,s}$ non-trivially covers $\tau / G_{1, 1}$. Therefore, in the poset of coverings of almost regular maps of the form $\tau / G_{k,s}$ (see \textsc{Case 1}) and $\tau / H_{l,s}$ the quotient of $\tau$ by $G_{1, 1}$ is the unique minimal element. 

\textsc{Sub-case 2:} assume condition (\ref{Eq:NecCond2}) is satisfied. This is equivalent of saying that both pairs $a_1/g_1$, $(2 a_1 + a_2)/g_2$ and $b_2/g_1$, $(2b_1 + b_2)/g_2$ consists of integers with the same parity. As we shown above, this is also equivalent to $\Delta / (g_1 g_2) \in \mathbb Z^*$.

If $k$ is even, then we can run verbatim the same arguments as in \textsc{Sub-case 1}, with the same conclusion. Hence we can assume that $k$ is odd. 

But if $k$ is odd and condition (\ref{Eq:NecCond2}) is met, then the pair of numbers
$$
n_2' = \frac{b_2}{2g_1} k + \frac{2b_1 + b_2}{2g_2}, \quad m_2' = - \frac{a_2}{2 g_1} k - \frac{2a_1 + a_2}{2g_2}
$$  
are necessarily full integers, and moreover provides a partial solution to (\ref{Eq:AuxEq}). Therefore, we can write down a complete solution to (\ref{Eq:AuxEq}):
\begin{equation*}
\begin{aligned}
&n_2 = \frac{b_2}{2g_1} k + \frac{2b_1 + b_2}{2g_2} + \frac{2b_1 + b_2}{g_2}s = \frac{b_2}{2g_1} k + \frac{2b_1 + b_2}{2g_2} (2s+1), \\
&m_2 = - \frac{a_2}{2 g_1} k - \frac{2a_1 + a_2}{2g_2} - \frac{2a_1 + a_2}{2g_2}s = - \frac{a_2}{2 g_1} k - \frac{2a_1 + a_2}{2g_2} (2s+1),
\end{aligned}
\end{equation*}
where $s \in \mathbb Z$.

Substituting this solution in either of the equations in (\ref{Eq:Case2SecondEq}), we obtain
$$
d = - \frac{\Delta}{2g_1} k - \frac{\Delta}{2 g_2} (2s+1),
$$
where both $k$ and $2s+1$ are some odd integers.

Similarly to as we did before, concluding \textsc{Sub-case 2} of \textsc{Case 2}, we see that, provided condition (\ref{Eq:NecCond2}) is satisfied and $k$ is odd (the case $k$ was already discussed), $\tau_{\bold{u}, \bold{v}}$ is an almost regular Archimedean cover of $\tau_{\bold{a}, \bold{b}}$ and is of the second type in (\ref{Eq:TwoTypes}) if and only if
$$
\left<\bold{u}, \bold{v}\right> = \left<\left(\frac{\Delta}{g_1}(2l+1), 0\right), \left(\frac{\Delta}{2 g_1} (2l + 1) + \frac{\Delta}{2 g_2} (2s + 1), - \frac{\Delta}{g_2} (2s+1)\right)\right> =: F_{l,s},
$$ 
where $l, s \in \mathbb Z$.

Finally, comparing the groups
$$
G_{1,1} = \left<\left(\frac{\Delta}{g_1}, 0\right), \left(\frac{\Delta}{g_2} , -2 \frac{\Delta}{g_2}\right)\right> \text{ and }F_{0,0}= \left<\left(\frac{\Delta}{g_1}, 0\right), \left(\frac{\Delta}{2g_1} + \frac{\Delta}{2g_2}, - \frac{\Delta}{g_2}\right)\right>,
$$ 
we conclude that the almost regular map $\tau / G_{1,1}$ non-trivially covers the almost regular map $\tau/ F_{0,0}$. 

Therefore, summing up the results of \textsc{Case 1} and \textsc{Case 2}, we obtain that if condition (\ref{Eq:NecCond2}) is satisfied, then $\tau / F_{0,0}$ is a minimal almost regular Archimedean map that covers $\tau_{\bold{a}, \bold{b}}$. Otherwise, $\tau / G_{1,1}$ is a minimal almost regular cover. In both cases these minimal covers are unique elements in the corresponding posets of possible almost regular covers. This almost finishes the proof of Theorem~\ref{Thm:33344}. The only thing that is left to check is that in \textsc{Sub-case 2} of \textsc{Case 2}, provided $k$ is odd, the number of fundamental regions of $\left<\bold{a}, \bold{b}\right>$ one should glue together to obtain the fundamental region of $F_{0,0}$ is equal to 
$$
\left|\frac{\Delta}{g_1 g_2}\right|
$$
(note that this is an integer). This computation is similar to the one found at the end of \textsc{Case 1}, and is thus omitted; this completes the proof of Theorem~\ref{Thm:33344}.
\end{proof}


\end{document}